\documentclass{article}
\pdfoutput=1
\usepackage[utf8]{inputenc}
\usepackage{amsmath}
\usepackage{amsthm}
\usepackage{amsfonts}
\usepackage{mathtools}
\usepackage{wasysym}
\usepackage{MnSymbol}
\usepackage{thmtools}
\usepackage{stmaryrd}
\usepackage[letterpaper,margin=1in]{geometry}    			 
\usepackage[english]{babel}				
\usepackage{bookmark}
\usepackage{url}					
\usepackage[T1]{fontenc}
\usepackage{xspace}		
\usepackage{fancyhdr}
\usepackage{enumerate}
\usepackage{mathrsfs}
\usepackage{graphicx}
\usepackage{soul,color}
\usepackage{mathtools}
\usepackage{tikz-cd}
\usepackage{tikz}
\usepackage[style=alphabetic, maxbibnames=99]{biblatex}
\usepackage{csquotes}
\usepackage{musicography}
\usepackage{chngcntr}
\usepackage[bbgreekl]{mathbbol}
\counterwithin{equation}{section}
\usepackage{spectralsequences}
\DeclareSymbolFontAlphabet{\mathbb}{AMSb}
\DeclareSymbolFontAlphabet{\mathbbl}{bbold}

\renewcommand{\epsilon}{\varepsilon}
\renewcommand{\rho}{\varrho}
\renewcommand{\phi}{\varphi}
\newcommand{\NN}{\ensuremath{\mathbb{N}}\xspace}
\newcommand{\ZZ}{\ensuremath{\mathbb{Z}}\xspace}
\newcommand{\QQ}{\ensuremath{\mathbb{Q}}\xspace}

\newcommand{\CC}{\ensuremath{\mathbb{C}}\xspace}
\newcommand{\FF}{\ensuremath{\mathbb{F}}\xspace}
\newcommand{\TT}{\ensuremath{\mathbb{T}}\xspace}

\newcommand{\DD}{\ensuremath{\mathbbl{\Delta}}\xspace}
\newcommand{\Sp}{\mathcal{S}p}
\newcommand{\tc}{\ensuremath{\operatorname{TC}}}
\newcommand{\thh}{\ensuremath{\operatorname{THH}}}
\newcommand{\tp}{\ensuremath{\operatorname{TP}}}

\newcommand{\colim}{\operatorname{colim}}
\newcommand{\pic}{\mathrm{Pic}}

\newtheorem{thm}{Theorem}[section]
\newtheorem{prop}[thm]{Proposition}
\newtheorem{lem}[thm]{Lemma}
\newtheorem{cor}[thm]{Corollary}

\theoremstyle{remark}
\newtheorem{rem}[thm]{Remark}

\title{\texorpdfstring{$K$-Theory of Cuspidal Curves Over a Perfectoid Base And Formal Analogues}{K-Theory of Cuspidal Curves Over a Perfectoid Base And Formal Analogues}}
\author{Noah Riggenbach}
\date{}

\begin{document}

\maketitle

\begin{abstract}
    In this paper we continue the work, started in \cite{Me}, of using the recent advances in algebraic $K$-theory to extend computations done in characteristic $p$ to the mixed characteristic setting using perfectoid rings. We extend the work of Hesselholt-Nikolaus in \cite{Hesselholt_Nikolaus} on the algebraic $K$-Theory of cuspidal curves. We consider both cuspidal curves and the $p$-completion of cuspidal curves. Along the way we also study the $K$-theory of the $p$-completed affine line over a perfectoid ring.
\end{abstract}

\tableofcontents

\section{Introduction}
The purpose of this paper is to use the recent developments in algebraic $K$-theory and topological cyclic homology to extend the calculation in \cite{Hesselholt_Nikolaus} to the mixed characteristic setting. Our setup is as follows: let $R$ be a perfectoid ring, and let $a,b\geq 2$ be coprime integers. Then the class of rings we will be studying are the coordinate rings of the cuspidal curve $A:=R[x,y]/(y^a-x^b)$. For this ring we have that $R\to A\to R$ is the identity, where the first map is the inclusion of constant terms and the second map is the quotient map $x,y\mapsto 0$. In particular \[K(A)\simeq K(R)\times K(A,(x,y))\] and our goal is to compute the second term in this equivalence. For $R$ a perfect $\FF_p$-algebra this is \cite{Hesselholt_Nikolaus} and more generally for a regular $\FF_p$-algebra this is \cite{Hesselholt_cusps} subject to a conjecture proven in \cite{angeltveit2019picard}. In addition \cite{angeltveit2019picard} computes the sizes of the homotopy groups of the second terms in the case $R=\ZZ$, but the present paper is the first complete mixed characteristic result. For a discussion of perfectoid rings see \parencite[Section 3]{BMS1}. Examples of perfectoid rings include perfect $\FF_p$-algebras, the $p$-completion of rings of integers $\mathcal{O}_F$ for $F/\QQ$ algebraic containing $\QQ(\zeta_{p^\infty})$, and $\mathcal{O}_{\CC_p}$ where $\CC_p=\widehat{\overline{\QQ}}_p$. The limit of perfectoid rings is again perfectoid as well. 

In order to state the computation, we first recall some notation. Let $\widehat{\ZZ}$ be the profinite completion of $\ZZ$. In other words $\widehat{\ZZ}\cong \prod_{p\in \mathbb{P}}\ZZ_p$. For a functor $F:\mathrm{Sch}^{op}\to \Sp$ denote by $F(-;\widehat{\ZZ})$ the functor which is $F$ composed with the profinite completion functor $\widehat{(-)}:\Sp\to\Sp$. We also need some notation from \cite{Hesselholt_Nikolaus} which we will recall now: let $l(a,b,m)=|\{(i,j)\in \ZZ_{\geq 1}^2| ai+bj=m\}|$, and $s(a,b,r,p,m')$ be the unique positive integer if it exists such that $l(a,b,m'p^{s(a,b,r,p,m')-1})\leq r<l(a,b,m'p^{s(a,b,r,p,m')})$, and is zero if no such positive integer exists. Let $S(a,b,r)$ be the set of integers $m$ such that $l(a,b,m)\leq r$. Let $J_p$ be the set of non-negative integers coprime to $p$. Assume without loss of generality that $b\in J_p$, and let $m':=m/p^{v_p(m)}\in J_p$ and $a'=a/p^{v_p(a)}\in J_p$. Finally, let $h=h(a,b,r,p,m')$ be the piecewise function \[h(a,b,r,p,m')=\begin{cases}
s   & \textrm{ if } a',b\nmid m'\\
\min\{s, v_p(a)\} &\textrm{ if }a'\mid m\textrm{ and }b\nmid m'\\
0   &\textrm{ otherwise}
\end{cases}
\]
and for $r<0$ this function is zero.
\begin{thm}\label{thm: main computation}
Let $R$ be a perfectoid ring such that $K(R[t];\widehat{\ZZ})\simeq K(R;\widehat{\ZZ})$, such as any perfectoid ring containing a perfectoid valuation ring. Then \[K_{2r}(A,(x,y);\widehat{\ZZ})\cong \prod_{m'\in J_p}W_{h(a,b,r,p,m')}(R)\cong \mathbb{W}_S(R)/(V_a\mathbb{W}_{S/a}(R)+V_b\mathbb{W}_{S/b}(R))\] and the odd homotopy groups vanish.
\end{thm}
In particular taking $R$ to be a perfect $\mathbb{F}_p$ algebra we recover the results of \cite{Hesselholt_Nikolaus}. Our argument follows that of \cite{Hesselholt_Nikolaus} closely, replacing the relevant fact about perfect $\FF_p$ algebras with results from \cite{BMS1, BMS2, Me}. For a discussion of the generalized Witt vectors showing up here, see \parencite[Section 1]{Hesselholt_big_drw_cplx}.

A complication which does not show up in the characteristic $p$ cases is what the appropriate cuspidal curve to take is. For $k$ a ring of characteristic $p$, $A$ is $p$-complete and so the cuspidal curve over $\mathrm{Spec}(k)$ and $\mathrm{Spf}(k)$ are the same. On the other hand, for $R$ a general perfectoid ring they will be quite different, and in many ways it is much more natural to consider $\mathrm{Spf}(R)$ than $\mathrm{Spec}(R)$. We also compute the relative $K$-theory for $A^\wedge_p$, which turns out to give a similar answer. In the $p$-complete case we need a replacement for the assumption of homotopy invariance. We are able to produce such a result as long as we are willing to work over $\ZZ_p^{cycl}:=\ZZ_p[\zeta_{p^\infty}]^\wedge_p$. In order to get this computation we first need to compute the algebraic $K$-theory of the formal affine line. Recall that for an abelian group $G$ the Tate module of $G$ is the group $T_pG:=\hom(\ZZ/p^\infty, G)$.

\begin{thm}\label{thm: Formal affine invariance}
Let $R$ be a perfectoid $\ZZ_p^{cycl}$-algebra, and let $R_0:=R/R[p]$. Let $R\langle t\rangle:= R[t]^\wedge_p$. Then there are isomorphisms \[K_{2r}(R\langle t\rangle, (t);\widehat{\ZZ})\cong 0\] and  \[K_{2r+1}(R\langle t\rangle, (t);\widehat{\ZZ})\cong T_p(\pic(R_0\langle t\rangle))\oplus \left(1+\sqrt{pR_0}\langle t\rangle \right)^\wedge_p\] for $r\geq 0$, where $1+\sqrt{pR_0}\langle t\rangle \subseteq R_0\langle t\rangle$ is the subset of power series with constant term $1$ and all other coefficients in the radical of $p$.
\end{thm}

\begin{rem}
The proof of this statement uses the Henselian ridgidity properties of nil invariant $K$-theory proven in \cite{Clausen_Mathew_Morrow} to reduce to the same statement for topological cyclic homology, and then use the syntomic complexes $\ZZ_p(i)(R\langle t\rangle)$ from \cite{BMS2} and the $ku^\wedge_p$-module structure from \cite[Example 5.5]{k1_local_tr}. Elmanto has shown in \cite{Elden} that we should not expect affine invariance for topological cyclic homology. Note that $pt\in 1+\sqrt{pR_0}\langle t\rangle$ is not in the image of the multiplication by $p$ map\footnote{the group structure here is multiplication, so the multiplication by $p$ map is taking $p$-powers.} and so $(1+\sqrt{pR_0}\langle t\rangle)^\wedge_p$ is not zero. While this does match with Elmanto's result, it is interesting that this differs from the positive characteristic case where the relative $\tc$-groups vanish is nonnegative dimension, see Lemma~\ref{lem: homotopy invariance for perfect fp algebras}.
\end{rem}

\begin{rem}
For simplicity we have written the odd $K$-groups above as a direct sum. While this is true it is also slightly misleading since the decomposition is not natural. There is in fact a natural short exact sequence with these terms which is splittable, see Corollary~\ref{cor: odd k group calculation} for the precise statement.
\end{rem}
With this computation we are able to apply the work we did in the case of $A$ to the case of $A^\wedge_p$ in the case when we are working over $\ZZ_p^{cycl}$.
\begin{thm}\label{thm: main computation p-completed}
Let $R$ be a perfectoid $\ZZ_p^{cycl}$-algebra and $R_0:=R/R[p]$. Then for $r\geq 0$ we have an exact sequence \[
\begin{tikzcd}
0 \arrow[r] & {K_{2r+1}(A^\wedge_p,(x,y);\widehat{\ZZ})} \arrow[r]                          & T_p(\pic(R_0\langle t\rangle))\oplus \left(1+\sqrt{pR_0}\langle t\rangle \right)^\wedge_p \arrow[ld, out=-10, in=170] &   \\
            & \mathbb{W}_S(R)/(V_{a}\mathbb{W}_{S/a}(R)+V_{b}\mathbb{W}_{S/b}(R)) \arrow[r] & {K_{2r}(A^\wedge_p, (x,y);\widehat{\ZZ})} \arrow[r]                       & 0
\end{tikzcd}
\] where the map from the top to bottom rows comes from the boundary map in an exact triangle from the octahedron axiom.
\end{thm}

\begin{rem}
In order to get a computation of $K(A,(x,y))$ and $K(A^\wedge_p, (x,y))$ without the coefficients it remains to study the rationalization of $K$-theory. Unfortunately $(x,y)\subseteq A$ is not nilpotent, and so the usual simplification of the rational theory is unavailable. That being said, for the uncompleted case we will see in Section~\ref{sec: main pullback squares} that we still rationally have a pullback square comparing $K(A)_{\QQ}$ with $K(R[t])_\QQ$ and $\tc(A,R[t])_{\QQ}$. It would be interesting to know how close the latter term is to $\tc(A_{\QQ},R[t]_{\QQ})$. From \parencite[Theorem 2.1]{angeltveit2019picard}, recorded in this paper as Theorem~\ref{thm: Bm as a T spectrum}, we will see that on topological Hochschild homology this is made up of wedge sum of induced actions. Hence $\thh(A,R[t])$ will have a meromorphic Frobenius in the sense of \parencite[Definition 5.17.2]{Raskin} and the weight filtration is a candidate for \parencite[Corollary 5.18.7]{Raskin}.
\end{rem}

\textbf{Acknowledgements.}
I would like to thank Benjamin Antieau, Ayelet Lindenstrauss, Michael Mandell, Akhil Mathew, Matthew Morrow, Martin Speirs, and Yuri Sulyma for helpful conversations on the material in this paper. I would also like to thank Benjamin Antieau, Elden Elmanto, and Peter Marek for reading an earlier draft and for many helpful comments.

\section{The main pullback squares}\label{sec: main pullback squares}

In this section we will explain how to reduce the $K$-theory computations in Theorems~\ref{thm: main computation} and Theorem~\ref{thm: main computation p-completed} to calculations in topological cyclic homology. By the recent advances in topological cyclic homology, especially of perfectoid rings, produced in \cite{Nikolaus_Scholze}, \cite{BMS1}, and \cite{BMS2} we are then able to do these topological cyclic homology computations in later sections. For the reduction to topological cyclic homology we will follow the general strategy of \cite{Hesselholt_Nikolaus}.

To this end, consider for any ring $B$ the following square
\[
\begin{tikzcd}
{\mathrm{Spec}(B[t]/(t^{\min(a,b)}))} \arrow[d] \arrow[r] & {\mathrm{Spec}(B[t])} \arrow[d] \\
\mathrm{Spec}(B) \arrow[r]                                & \mathrm{Spec}(B[x,y]/(y^a-x^b))                              
\end{tikzcd}
\]
where the map $\mathrm{Spec}(B)\to \mathrm{Spec}(B[x,y]/(y^a-x^b))$ is the quotient map $x,y\mapsto 0$, and the map $\mathrm{Spec}(B[t])\to \mathrm{Spec}(B[x,y]/(y^a-x^b))$ is the inclusion $x\mapsto t^a$, $y\mapsto t^b$.

The square above is cartesian, the map $\mathrm{Spec}(B)\to \mathrm{Spec}(B[x,y]/(y^a-x^b))$ is a closed immersion, and the map $\mathrm{spec}(B[t])\to \mathrm{Spec}(B[x,y]/(y^a-x^b))$ is proper and an isomorphism away from $\mathrm{Spec}(B)$. In other words the square is a cdh-square. We may now apply \parencite[Theorem E]{Land_Tamme} to get a pullback square
\[
\begin{tikzcd}
K^{inv}(B[x,y]/(y^a-x^b)) \arrow[d] \arrow[r] & K^{inv}(B)\arrow[d]\\
K^{inv}(B[t]) \arrow[r] & K^{inv}(B[t]/(t^{\min\{a,b\}}))
\end{tikzcd}
\]
for all rings $B$. The map $K^{inv}(B)\to K^{inv}(B[t]/(t^{\min\{a,b\}}))$ is split by the map which sends $t\mapsto 0$ which is an equivalence by the Dundas Goodwillie McCarthy theorem \cite[Theorem 7.0.0.1]{Dundas_Goodwillie_McCarthy}. Thus the right vertical map in the above pullback square is an equivalence and therefore so too is the right hand vertical map. In particular we get a pullback square 
\[
\begin{tikzcd}
K(B[x,y]/(y^a-x^b))\arrow[d]\arrow[r] & \tc(B[x,y]/(y^a-x^b))\arrow[d]\\
K(B[t])\arrow[r] & \tc(B[t])\\
\end{tikzcd}
\]
for all rings $B$.

For $B=R$ taking the profinite completion of the above pullback square gives the desired translation for Theorem~\ref{thm: main computation}. For Theorem~\ref{thm: main computation p-completed} we will handle completion at $p$ and at $n\in J_p$ separately and will only need to reduce to a topological cyclic homology in the first case. For this notice that taking $B=R/p^k$ then gives a pullback square 
\[
\begin{tikzcd}
K(R/p^k[x,y]/(y^a-x^b);\ZZ_p)\arrow[d]\arrow[r] & \tc(R/p^k[x,y]/(y^a-x^b);\ZZ_p)\arrow[d]\\
K(R/p^k[t];\ZZ_p)\arrow[r] & \tc(R/p^k[t];\ZZ_p^k)\\
\end{tikzcd}
\]
for all $k\in \NN$. Taking the limit as $k\to \infty$ and using \cite[Theorem A]{Clausen_Mathew_Morrow} then gives a pullback square
\begin{equation}\label{eqn: main pullback square}
\begin{tikzcd}
K(A^\wedge_p;\ZZ_p)\arrow[d]\arrow[r] & \tc(A^\wedge_p;\ZZ_p)\arrow[d]\\
K(R\langle t\rangle;\ZZ_p)\arrow[r] &\tc(R\langle t\rangle;\ZZ_p)
\end{tikzcd}
\end{equation}
where as in Theorem~\ref{thm: Formal affine invariance} $R\langle t\rangle:= R[t]^\wedge_p$.

For the pullback square 
\[
\begin{tikzcd}
K(A;\widehat{\ZZ})\arrow[d]\arrow[r] & \tc(A;\widehat{\ZZ})\arrow[d]\\
K(R[t];\widehat{\ZZ})\arrow[r] &\tc(R[t];\widehat{\ZZ})\\
\end{tikzcd}
\]
\cite[Theorem 1.5]{AMM} shows that for $R$ a perfectoid algebra over a perfectoid valuation ring $K(R;\ZZ_p)\to K(R[t];\ZZ_p)$ will be an isomorphism. Combined with results of Weibel this is enough to show that the left vertical fiber is what we are trying to compute. We go into more detail about this reduction in Section~\ref{sec: final computation}. For Theorem~\ref{thm: main computation p-completed} the square above is sufficient once we have Theorem~\ref{thm: Formal affine invariance}. Theorem~\ref{thm: Formal affine invariance} will also be proven by reducing to a topological cyclic homology computation, but this is a slightly different reduction than above and will be done at the beginning of Section~\ref{sec: affine line}.

\section{The topological cyclic homology computation}\label{sec: tc calculation}

In this section we will compute the right hand side of Diagram~\ref{eqn: main pullback square}. Since everything in sight has bounded $p^\infty$-torsion this computation will work in both cases of interest. While we do end up computing $\tc(R\langle t\rangle;\ZZ_p)$ (at least in high degrees) in Subsection~\ref{ssec: tc affine line general}, it will be more convenient for the final computation to compute the fiber term $\tc(A,R[t];\ZZ_p)$ instead of computing $\tc(A;\ZZ_p)$ and using the long exact sequence on homotopy groups. We will be following the arguments in \cite{Hesselholt_Nikolaus} closely.

In order to do this we first recall the simplification to topological cyclic homology computations produced by Nikolaus and Scholze.
\begin{thm}[\cite{Nikolaus_Scholze}, Theorem II.4.11]\label{thm: ns main result}
Let $X$ be a genuine cyclotomic spectrum which has a bounded below underlying spectrum. Then there is a functorial equivalence \[\tc(X)\simeq \mathrm{fib}\left(\tc^{-}(X)\xrightarrow{(\phi_p-can)_p}\prod_{p\in \mathbb{P}}\tp(X;\ZZ_p)\right)\]
\end{thm}

Here the map $can$ is the composition \[X^{h\TT}\simeq \left(X^{hC_p}\right)^{h\TT/C_{p}}\to\left(X^{tC_p}\right)^{h\TT/C_p}\simeq \left(X^{t\TT}\right)^\wedge_p\] where the middle map is the homotopy $\TT/C_p$-fixed points of the map $X^{hC_p}\to X^{tC_p}$ which define the Tate construction and the last equivalence is \parencite[Lemma II.4.2]{Nikolaus_Scholze}. By Using the universal properties of the canonical maps discussed in \parencite[Theorem I.4.1]{Nikolaus_Scholze} and of $p$-completion this can be shown to agree with the map $X^{h\TT}\to X^{t\TT}$ followed by the $p$-completion map. The Frobenius map $\phi_p$ is defined as the composition \[X^{h\TT}\xrightarrow{(\phi_p^X)^{h\TT}}\left(X^{tC_p}\right)^{h\TT/C_p}\simeq \left(X^{t\TT}\right)^\wedge_p\] where $\phi_p^X:X\to X^{tC_p}$ is the $\TT\cong \TT/C_p$-equivariant Frobenius map given by the definition of cyclotomic spectra. 

For $X\to Y$ a map of bounded below cyclotomic spectra we can calculate $\tc(X,Y)$ using the same fiber sequences and definitions as above since taking the cofiber of $X\to Y$ will commute with all limits and colimits. In Subsection~\ref{ssec: relative thh} we describe this term equivariantly and also describe the canonical and Frobenius maps in terms of the canonical and Frobenius maps on $\thh(R;\ZZ_p)$. In Subsection~\ref{ssec: relative tc- and tp} we compute the topological negative cyclic homology and periodic homology from the equivariant description from Subsection~\ref{ssec: relative thh}. Finally we put these pieces together to get the relevant topological cyclic homology computation in Subsection~\ref{ssec: relative tc}.

\subsection{The relative topological Hochschild homology}\label{ssec: relative thh}

We will mostly be using formal properties of topological Hochschild homology to do these computations. First note that as a functor $\thh(-):\mathrm{Alg}_{\mathbb{E}_1}(\Sp)\to \mathrm{CycSp}$ is symmetric monoidal by \parencite[Section IV.2]{Nikolaus_Scholze}. In particular if we have a pointed monoid $\Pi$ then \[\thh(R(\Pi))\simeq \thh(R\wedge \mathbb{S}(\Pi))\simeq \thh(R)\wedge \Sigma^{\infty}B^{cy}(\Pi)\] where $B^{cy}(\Pi)$ is the geometric realization of the simplicial space $B_n^{cy}(\Pi)=\Pi^{\wedge n+1}$ and face maps given by the usual formula in the Hochschild complex. We refer the reader to \parencite[Sections 3.1 and 3.2]{Speirs_truncated_polynomials} for a nice review of this decomposition. In particular recall that by the definition of the symmetric monoidal product on the target the canonical map and the Frobenius maps will decompose into the tensor product of the corresponding maps followed by the map witnessing that $(-)^{tC_p}$ is lax-monoidal. For example in diagrams this means that \[\begin{tikzcd}
\thh(R)\wedge \Sigma^\infty B^{cy}(\Pi) \arrow[rrdd, "\phi_p", bend right] \arrow[rr, "\phi_p^{\thh(R)}\wedge id"] &  & \thh(R)^{tC_{p}}\wedge \Sigma^\infty B^{cy}(\Pi) \arrow[d, "id\wedge \phi_p^{\thh(\mathbb{S}(\Pi))}"] \\
                                                                                                                   &  & \thh(R)^{tC_p}\wedge \left(\Sigma^{\infty}B^{cy}(\Pi)\right)^{tC_{p}} \arrow[d, "l"]                  \\
                                                                                                                   &  & \left(\thh(R)\wedge \Sigma^{\infty}B^{cy}(\Pi)\right)^{tC_{p}}                                       
\end{tikzcd}\]
commutes.

In the case at hand, we have that $R[t]\cong R\wedge \Pi_t$ where $\Pi_t=\{0, t, t^2,\ldots\}$ is the free pointed monoid on one generator, and $A\cong R\wedge \Pi_{t^a, t^b}$ where $\Pi_{t^a,t^b}=\{0, t^a, t^b, t^{a+b}, \ldots\}$ is the submonoid of $\Pi_t$ generated by $t^a$ and $t^b$. Thus \[\thh(A,R[t])\simeq \thh(R)\wedge \mathrm{fib}\left(\Sigma^{\infty}B^{cy}(\Pi_{t^a,t^b})\to \Sigma^{\infty}B^{cy}(\Pi_t)\right)\simeq \thh(R)\wedge \Sigma^{\infty -1} B^{cy}(\Pi_t/\Pi_{t^a,t^b})\] where the second equivalence comes from switching the fiber to a cofiber and computing the quotient level-wise since in each simplicial degree the induced map is injective.

To study this space, we will use a decomposition used in \cite{angeltveit2019picard, Hesselholt_Nikolaus} for their computations. Specifically let $B_m$ be the subspace of $B^{cy}(\Pi_t/\Pi_{t^a,t^b})$ given in simplicial degree $n$ by the terms $x_0\wedge \ldots \wedge x_n$ with $\sum_{i=0}^n \deg(x_i)=m$. The following was conjectured by Hesselholt in \parencite[Conjecture B]{Hesselholt_cusps} and was recently proven by Angeltveit. The statement in \cite{angeltveit2019picard} is stronger than what is stated below since they need a genuine description of the $\TT$ action. We only need the Borel equivariant version of their statement, and so will only record that here. Recall that $\CC(n)$ is defined to be the $\TT$ representation on $\CC$ given by $z$ acting by multiplication by $z^n$.
\begin{thm}[Theorem 2.1, \cite{angeltveit2019picard}]
There is a Borel $\TT$-equivariant identification of $\Sigma^\infty B_m$ with the total cofiber of the square 
\[\begin{tikzcd}\label{thm: Bm as a T spectrum}
{S^{\lambda(a,b,m)}\wedge (\TT/C_{m/ab})_+} \arrow[r] \arrow[d] & {S^{\lambda(a,b,m)}\wedge (\TT/C_{m/b})_+} \arrow[d] \\
{S^{\lambda(a,b,m)}\wedge (\TT/C_{m/a})_+} \arrow[r]            & {S^{\lambda(a,b,m)}\wedge (\TT/C_{m})_+}            
\end{tikzcd}\]
where the maps are all the canonical projections, $\TT/C_{m/c}:=\emptyset$ if $c\nmid m$, and $\lambda(a,b,m)$ the $\TT$ representation $\bigoplus_{n\in (cm/a, dm/b)\cap \ZZ}\CC(n)$ where $c$ and $d$ are some integers such that $ad-bc=1$.
\end{thm}

Notice that $\frac{dm}{b}-\frac{cm}{a}=\frac{adm-bcm}{ab}=\frac{m}{ab}$, and so $\dim(\lambda(a,b,m))\sim m$ and for any given dimension $n$ there are only finitely many $m$ such that $\dim(\lambda(a,b,m))\leq n$. Thus $B_m$ has connectivity $o(m)$ and so \[\thh(R)\wedge \bigvee_{m\in \ZZ_{\geq 1}}\Sigma^{\infty-1}B_m\simeq \bigvee_{m\in \ZZ_{\geq 1}}\Sigma^{-1}\thh(R)\wedge B_m\simeq \prod_{m\in \ZZ_{\geq 1}}\Sigma^{-1}\thh(R)\wedge B_m\] where the last equivalence comes from the connectivity bounds.

\begin{rem}
The statement in \cite{angeltveit2019picard} also requires a $p$-completion for the equivalence. Notice, however, that \parencite[Theorem 2.1]{angeltveit2019picard} works at all primes. Using either of \cite{Larsen, Buenos_Aires_Group} we can get the rational statement as well in a similar fashion to \parencite[Theorem 1]{Hesselholt_Nikolaus}. So by using an arithmetic fracture square argument the Theorem is true as stated.
\end{rem}

From the above description we see that for any $H\leq \TT$ we have \[\thh(A,R[t])^{hH}\simeq \Sigma^{-1}\prod_{m\in \ZZ_{\geq 1}}\left(\thh(R)\wedge B_m\right)^{hH}\] and \[\thh(A,R[t])^{tH}\simeq \Sigma^{-1}\prod_{m\in \ZZ_{\geq 1}}\left(\thh(R)\wedge B_m\right)^{tH}\] the first equivalence coming from limits commuting with other limits and the second from homotopy $H$ orbits commuting with the wedge sum and not changing the connectivity bounds. From these products we can see that the canonical map will decompose as the product of the canonical maps on each piece. From the description of the Frobenius above and by checking on connected components we see that $\phi^{B^{cy}(\Pi_t/\Pi_{t^1,t^b})}_p$ sends the summand $B_m$ to the factor $B_{pm}^{tC_p}$ in the product. Thus that Frobenius on the relative topological Hochschild Homology also sends $\thh(R)\wedge B_m$ to $\left(\thh(R)\wedge B_{pm}\right)^{tC_p}$. A slightly modified version of \parencite[Lemma 2]{Hesselholt_Nikolaus} then shows that the maps \[\thh(R)^{tC_p}\wedge B_{m/p}\xrightarrow{id\wedge \phi_p^{B^{cy}(\Pi_t/\Pi_{t^a,t^b})}}\thh(R)^{tC_p}\wedge B_m^{tC_p}\xrightarrow{l}\left(\thh(R)\wedge B_m\right)^{tC_p}\] are equivalences, where $B_{m/p}=0$ if $p\nmid m$. Consequently up to equivalence we may use the canonical map term-wise and up to equivalence we may use the Frobenius on $\thh(R)$ instead of the Frobenius on $\thh(A,R[t])$.

Finally, we restrict our attention to $R$ perfectoid with respect to a prime $p$. We refer the reader to \parencite[Section 5.1]{Me} for the relevant properties of perfectoid rings, which is itself a review of the relevant main results from \cite{BMS1, BMS2}. Applying these results to the case at hand gives us the following.
\begin{cor}
Let $B_m$ be as above and $R$ a perfectoid ring. Then there is a functorial identification of $\thh(R;\ZZ_p)\wedge B_m$ as a Borel $\TT$-spectrum with the total cofiber of the square 
\[\begin{tikzcd}
{\Sigma^{2l(a,b,m)}\thh(R;\ZZ_p)\wedge (\TT/C_{m/ab})_+} \arrow[r] \arrow[d] & {\Sigma^{2l(a,b,m)}\thh(R;\ZZ_p)\wedge (\TT/C_{m/b})_+} \arrow[d] \\
{\Sigma^{2l(a,b,m)}\thh(R;\ZZ_p)\wedge (\TT/C_{m/a})_+} \arrow[r]            & {\Sigma^{2l(a,b,m)}\thh(R;\ZZ_p)\wedge (\TT/C_{m})_+}
\end{tikzcd}\]
where the maps are as in Theorem~\ref{thm: Bm as a T spectrum} and $l(a,b,m):=\frac{1}{2}\dim(\lambda(a,b,m))=|\{(i,j)\in \ZZ_{\geq 1}\times \ZZ_{\geq 1}| ai+bj=m\}|$.
\end{cor}
\begin{proof}
The only part of this that does not follow immediately from smashing the square in Theorem~\ref{thm: Bm as a T spectrum} with $\thh(R;\ZZ_p)$ is that the representation spheres can be trivialized. This can in fact be done for $\thh(A;\ZZ_p)$ for any $\mathbb{E}_\infty$ $R$-algebra $A$ by \parencite[Proposition 4.2]{Me}.
\end{proof}

Recall that we assumed without loss of generality in the beginning that $b\in J_p$, in other words that $p\nmid b$. Following \cite{Hesselholt_Nikolaus} we will now give a more explicit form for $\thh(R;\ZZ_p)\wedge B_m$.

\begin{cor}\label{cor: thh(R)wedge B_m}
Let $R$ be a perfectoid ring, $B_m$ as above. Then
\begin{enumerate}
    \item if both $a$ and $b$ do not divide $m$ then \[\thh(R;\ZZ_p)\wedge B_m\simeq \Sigma^{2l(a,b,m)}\thh(R;\ZZ_p)\wedge (\TT/C_m)_+;\]
    \item if $a\mid m$ and $b\nmid m$ then \[\thh(R;\ZZ_p)\wedge B_m\simeq \Sigma^{2l(a,b,m)}\mathrm{cof}\left(\thh(R;\ZZ_p)\wedge (\TT/C_{m/a})_+\to \thh(R;\ZZ_p)\wedge (\TT/C_{m})_+\right)\] and;
    \item if $b\mid m$ then $\thh(R;\ZZ_p)\wedge B_m\simeq 0$.
\end{enumerate}
\end{cor}
\begin{proof}
For (1) and (2) this is how we defined the pushout square in Theorem~\ref{thm: Bm as a T spectrum}. For (3) either $b\mid m$ and $a\nmid m$ or $a,b\mid m$. If $a\nmid m$ then we also have a cofiber sequence by definition $\thh(R;\ZZ_p)\wedge B_m\simeq \Sigma^{2l(a,b,m)}\mathrm{cof}\left(\thh(R;\ZZ_p)\wedge(\TT/C_{m/b})_+\to \thh(R;\ZZ_p)\wedge (\TT/C_m)_+\right)$. Since $B_m$ if finite as an underlying space is follows that $\thh(R;\ZZ_p)\wedge B_m$ is $p$-complete. Then since $p\nmid b$ the map $\TT/C_{m/b}\to \TT/C_m$ is a $p$-adic equivalence. Finally if both $a,b\mid m$ then $B_m$ comes from simplicies of $B^{cy}(\Pi_t)$ which are in the image of $B^{cy}(\Pi_{t^a,t^b})$ by definition and so will be equivariantly contractable. 
\end{proof}

\subsection{Topological negative cyclic and periodic homology for the relative term}\label{ssec: relative tc- and tp}. 

From the last section we see that in order to compute the topological negative cyclic homology and topological periodic homology of $A$ we will need to take homotopy fixed points and Tate constructions of several Borel $\TT$-spectra of the form $\thh(R;\ZZ_p)\wedge (\TT/C_m)_+$. In order to do this we recall the following.

\begin{prop}[Proposition 3, \cite{Hesselholt_Nikolaus}]
Let $G$ be a compact Lie group. Let $H\subseteq G$ be a closed subgroup, let $\lambda=T_H(G/H)$ be the tangent space at $H=eH$ with the adjoint left $H$-action, and let $S^\lambda$ be the one-point compactification of $\lambda$. For every spectrum with $G$-action $X$, there are canonical natural equivalences \[\left(X\wedge \left(G/H\right)_+\right)^{hG}\simeq \left(X\wedge S^\lambda\right)^{hH},\] \[\left(X\wedge \left(G/H\right)_+\right)^{tG}\simeq \left(X\wedge S^\lambda\right)^{tH}\]
\end{prop}

From the discussion of the Frobenius in the previous section we already know that on these terms up to equivalence we may use the Frobenius on $\thh(R;\ZZ_p)$ instead. For what the canonical map is doing on these terms, the universal property discussed in \parencite[Theorem I.4.1]{Nikolaus_Scholze} guarantees that the canonical maps commute with the above isomorphisms. We also have the when $G=\TT$ and $H=C_m$ that the induced action on the tangent space is trivial and so \[\left(X\wedge (\TT/C_m)_+\right)^{h\TT}\simeq \Sigma X^{hC_m}\] and \[\left(X\wedge (\TT/C_m)_+\right)^{t\TT}\simeq \Sigma X^{tC_m}\] for all Borel $\TT$-spectra $X$.

\begin{lem}\label{lem: tc- computation}
Let $R$ be a perfectoid ring. Then \begin{align*}\tc^{-}_{2r}(A,R[t];\ZZ_p)\cong &\left(\prod_{m\in a\ZZ_{\geq 1}\setminus b\ZZ_{\geq 1}, l(a,b,m)<r}W_{v_p(m)}(R)/V_{v_p(a)}W_{v_p(m)-v_p(a)}(R)\right)\\
&\times \left(\prod_{m\in a\ZZ_{\geq 1}\setminus b\ZZ_{\geq 1}, l(a,b,m)\geq r}W_{v_p(m)+1}(R)/V_{v_p(a)}W_{ v_p(m)-v_p(a)+1}(R)\right)\\
&\times \left(\prod_{m\in \ZZ_{\geq 1}\setminus (a\ZZ_{\geq 1}\cup b\ZZ_{\geq 1}), l(a,b,m)<r}W_{v_p(m)}(R)\right)\\
                                    & \times \left(\prod_{m\in \ZZ_{\geq 1}\setminus (a\ZZ_{\geq 1}\cup b\ZZ_{\geq 1}), l(a,b,m)\geq r}W_{v_p(m)+1}(R)\right)
\end{align*} 
where $v_p$ is the $p$-adic valuation and $V_k: W_n(R)\to W_{n+k}(R)$ is the Verschiebung map. The odd homotopy groups vanish.
\end{lem}
\begin{proof}
From the last section we have that $\thh(A,R[t];\ZZ_p)\simeq \prod_{m\geq 1}\Sigma^{-1}\thh(R;\ZZ_p)\wedge B_m$ where the factors are equivariantly described by Corollary~\ref{cor: thh(R)wedge B_m}. Consequently \[\tc^{-}(A,R[t];\ZZ_p)\simeq \prod_{m\geq 1}\Sigma^{-1}\left(\thh(R;\ZZ_p)\wedge B_m\right)^{h\TT}\] and so we need only compute $\left(\thh(R;\ZZ_p)\wedge B_m\right)^{h\TT}$ for each $m\geq 1$. Using Corollary~\ref{cor: thh(R)wedge B_m} we will break this into cases depending on $m$.

Case One: $b\mid m$. Then $\thh(R;\ZZ_p)\wedge B_m\simeq 0$ as Borel $\TT$-spectra and so their homotopy fixed points will vanish.

Case Two: $a\nmid m$ and $b\nmid m$. Then By Corollary~\ref{cor: thh(R)wedge B_m} $\thh(R;\ZZ_p)\wedge B_m\simeq \Sigma^{2l(a,b,m)}\thh(R;\ZZ_p)\wedge (\TT/C_m)_+$ and so by the above proposition \[\left(\thh(R;\ZZ_p)\wedge B_m\right)^{h\TT}\simeq \Sigma^{2l(a,b,m)+1}\thh(R;\ZZ_p)^{hC_m}\simeq \Sigma^{2l(a,b,m)+1}\thh(R;\ZZ_p)^{hC_{v_p(m)}}\] where the second equivalence comes from the fact that $C_{v_p(m)}\hookrightarrow C_m$ is a $p$-adic equivalence. These then contribute the even homotopy groups by \parencite[Corollary 5.8]{Me}. 

Case Three: $a\mid m$ and $b\nmid m$. Then by the above we have a cofiber sequence \[\Sigma^{2l(a,b,m)+1}\thh(R;\ZZ_p)^{hC_{\frac{m}{a}}}\to \Sigma^{2l(a,b,m)+1}\thh(R;\ZZ_p)^{hC_{m}}\to \thh(R;\ZZ_p)\wedge B_m\] where the first map comes from the projection $\TT/C_{\frac{m}{a}}\to \TT/C_m$. The induced map on the homotopy fixed points was shown in \parencite[Lemma 5.13]{Me} to agree with the Verschiebung map. The result follows.
\end{proof}

A similar argument gives the following computation as well.

\begin{lem}
Let $R$ be a perfectoid ring. Then 
\begin{align*}
    \tp_{2r}(A,R[t];\ZZ_p)\cong &\prod_{m\in a\ZZ_{\geq 1}\setminus b\ZZ_{\geq 1}}W_{v_p(m)}(R)/V_{v_p(a)}W_{v_p(m)-v_p(a)}(R)\\
                                    &\times \prod_{m\in \ZZ_{\geq 1}\setminus(a\ZZ_{\geq 1}\cup b\ZZ_{\geq 1})}W_{v_p(m)}(R)
\end{align*}
where $v_p$ is the $p$-adic valuation. The odd homotopy groups vanish.
\end{lem}

\subsection{The relative topological cyclic homology computation}\label{ssec: relative tc}

From Theorem~\ref{thm: ns main result} to compute the topological cyclic homology we need only understand the topological negative cyclic homology, the topological periodic homology, and the maps between them. From the previous two subsections we have all the pieces, all that remains is to put them together. While Subsection~\ref{ssec: relative tc- and tp} computed the whole topological negative cyclic and periodic homologies, it will be much more convenient to work with the individual factors and put it all together later. In order to do this we will re-index the product as follows: for a given $m\in \ZZ_{\geq 1}$ let $m':=m/p^{v_p(m)}$ and let $J_p:=\ZZ_{\geq 1}\setminus p\ZZ_{\geq 1}$. Then $m'\in J_p$ and every $m$ can be written uniquely as an element $m'\in J_p$ and a natural number $\nu\in \ZZ_{\geq 0}$ recording the $p$-adic valuation of $m$. Thus the fiber sequence coming from Theorem~\ref{thm: ns main result} can be re-expressed as \[\tc(A,R[t];\ZZ_p)\to \prod_{m'\in J_p}\prod_{\nu = 0}^{\infty}\Sigma^{-1}\left(\thh(R;\ZZ_p)\wedge B_{m'p^\nu}\right)^{h\TT}\xrightarrow{can-\phi_p}\prod_{m'\in J_p}\prod_{\nu=0}^\infty\Sigma^{-1}\left(\thh(R;\ZZ_p)\wedge B_{m'p^\nu}\right)^{t\TT}\] and the canonical map respects both products, and the Frobenius respects the product over $J_p$ and sends the $\nu^{\textrm{th}}$ factor to the $(\nu+1)^{\textrm{st}}$ factor. 

Since limits commute with products it is then enough to compute the fiber of the map \[\prod_{\nu=0}^\infty \Sigma^{-1}\left(\thh(R;\ZZ_p)\wedge B_{m'p^\nu}\right)^{h\TT}\xrightarrow{can-\phi_p}\prod_{\nu=0}^\infty \Sigma^{-1}\left(\thh(R;\ZZ_p)\wedge B_{m'p^\nu}\right)^{t\TT}\] for all $m'\in J_p$. There are two cases to consider: either $b\nmid m'$ or $b\mid m'$. Since by assumption $v_p(b)=0$, $b\mid m$ if and only if $b\mid m'$. Then $b\mid m'$ implies $\thh(R;\ZZ_p)\wedge B_m$ is equivariantly contractable, so we need only consider the case when $b\nmid m'$. 

To handle this case, it is once again convenient to split this problem into two cases: when $a':= a/p^{v_p(a)}$ divides $m'$ and when it does not. The easier case is when $a'\nmid m'$, and so we will handle it first.

In this case $a\nmid m$ and so $\thh(R;\ZZ_p)\wedge B_m\simeq \Sigma^{2l(a,b,m)}\thh(R;\ZZ_p)\wedge (\TT/C_m)_+$. Thus the map we need to compute the fiber of becomes \[\prod_{\nu =0}^\infty \Sigma^{2l(a,b,m)}\thh(R;\ZZ_p)^{hC_{p^\nu}}\xrightarrow{can-\phi_p}\prod_{\nu=0}^\infty \Sigma^{2l(a,b,m)}\thh(R;\ZZ_p)^{tC_{p^\nu}}\] where $can:\thh(R;\ZZ_p)^{hC_p^\nu}\to \thh(R;\ZZ_p)^{tC_p^\nu}$ is the usual canonical map and \[\phi_p:\thh(R;\ZZ_p)^{hC_{p^\nu}}\to \Sigma^{l(a,b,pm)-2l(a,b,m)}\thh(R;\ZZ_p)^{tC_{p^{\nu+1}}}\] is up to an equivalence given by the homotopy $C_\nu$ fixed points of the usual Frobenius $\phi_p:\thh(R;\ZZ_p)\to \thh(R;\ZZ_p)^{tC_p}$. 

Recall that $s=s(a,b,r,p,m')$ is the function from \cite{Hesselholt_Nikolaus} which is the unique integer $s$ such that $l(a,b,mp^{s-1})\leq r<l(a,b,m'p^s)$ if such an integer exists, and is zero otherwise.
\begin{lem}\label{lem: tc calculation part 1}
Suppose that $a,b\nmid m'$. Then the fiber of the map \[\prod_{\nu =0}^\infty \Sigma^{2l(a,b,m)}\thh(R;\ZZ_p)^{hC_{p^\nu}}\xrightarrow{can-\phi_p}\prod_{\nu=0}^\infty \Sigma^{2l(a,b,m)}\thh(R;\ZZ_p)^{tC_{p^\nu}}\] has homotopy group $W_{s(a,b,r,p,m')}(R)$ in degree $2r$, and the odd homotopy groups vanish. 
\end{lem}
\begin{proof}
We have identified the homotopy groups of the source and target of this map, and so all that remains is to check on each homotopy group what the maps are doing. Since both the source and target have no odd homotopy groups the kernel of the degree $2r$ homotopy groups will be the homotopy groups of the fiber in degree $2r$ and the cokernel of the degree $2r$ homotopy groups will be the homotopy group of the fiber in degree $2r-1$. 

To compute this, note that $can:\thh(R;\ZZ_p)^{hC_{p^\nu}}\to \thh(R;\ZZ_p)^{tC_{p^\nu}}$ is $\thh(R;\ZZ_p)_{hC_{p^\nu}}$ which is connective. Thus when $r< l(a,b,m)$ the map \[\pi_{2r}(can):\pi_{2r}\left(\Sigma^{2l(a,b,m)}\thh(R;\ZZ_p)^{hC_{p^\nu}}\right)\to \pi_{2r}\left(\Sigma^{2l(a,b,m)}\thh(R;\ZZ_p)^{tC_{p^\nu}}\right)\] is an isomorphism. Then we have a map of short exact sequences
\[
\begin{tikzcd}
0 \arrow[d]                                                                                                            &  & 0 \arrow[d]                                                                                     \\
\prod_{\nu=s}^{\infty}W_{\nu}(R) \arrow[rr, "can-\phi"] \arrow[d]                                                      &  & \prod_{\nu =s}^{\infty}W_{\nu}(R) \arrow[d]                                                     \\
{\prod_{\nu=0}^\infty\pi_{2r}\left(\Sigma^{2l(a,b,m)}\thh(R;\ZZ_p)^{hC_{p^\nu}}\right)} \arrow[rr, "can-\phi"] \arrow[d] &  & {\prod_{\nu=0}^\infty\pi_{2r}\left(\Sigma^{2l(a,b,m)}\thh(R;\ZZ_p)^{tC_{p^\nu}}\right)} \arrow[d] \\
\prod_{\nu=0}^{s-1}W_{\nu+1}(R) \arrow[d] \arrow[rr, "\overline{can-\phi_p}"]                                            &  & \prod_{\nu=0}^{s-1}W_{\nu}(R) \arrow[d]                                                         \\
0                                                                                                                      &  & 0                                                                
\end{tikzcd}
\]
where the map $\overline{can-\phi_p}$ is the map induced by the map $can-\phi_p:\prod_{\nu=0}^{s-1}W_{\nu+1}(R)\to \prod_{\nu=0}^{s}W_{\nu}(R)$ followed by the quotient map $\prod_{\nu=0}^{s}W_{\nu}(R)\to \prod_{\nu=0}^{s-1}W_{\nu}(R)$. In addition the top term comes from the factors of $\prod_{\nu =0}^\infty \Sigma^{2l(a,b,m)}\thh(R;\ZZ_p)^{hC_{p^\nu}}$ which contribute negative homotopy groups of $\thh(R;\ZZ_p)^{hC_{p^\nu}}$ in degree $2r$ and so on these factors the canonical map is an isomorphism. We first show that the top map is an isomorphism and so the kernel and cokernel of the map in question and those of the bottom map agree.

Let $(-)_k:\prod_{\nu=s}^\infty W_\nu(R)\to W_k(R)$ be the canonical projection map for all $s\leq k$. If $x\in \ker(can-\phi_p)$ and $x\neq 0$ then there exists some $k$ such that $x_k\neq 0$. Chose $k$ to be the smallest such integer. Then $(can(x)-\phi_p(x))_k=can(x_k)-\phi_p(x_{k-1})=can(x_k)\neq 0$ and so $x$ cannot be in the kernel and $can-\phi_p$ is injective. On the other hand if $y\in \prod_{\nu=s}^\infty W_\nu(R)$ then we can build a pre-image inductively as $x_s=can^{-1}(y_s)$ and $x_{s+n}=can^{-1}(y_{s+n}+\phi_p(y_{s+n-1}))$. Consequently $can-\phi_p$ is an isomorphism on these terms as desired.

It remains to compute the kernel and cokernel of the bottom map. There are the terms which are contributed by non-negative homotopy groups of $\thh(R;\ZZ_p)^{hC_{p^\nu}}$ and so the Frobenius is an isomorphism by \parencite[Proposition 6.2]{BMS2} and the fact that homotopy fixed points preserve coconnectivity. Using similar notation as in the previous paragraph, we can inductively build a pre-image of any element $y\in \prod_{\nu=0}^{s-1}W_\nu(R)$ by taking $x_{s-1}=0$ and $x_{s-1-n}=\phi_p^{-1}(-y_{s-n}+can(x_{s-n}))$ for $1\leq n\leq s-1$. Note that $W_0(R)=0$ so we do not have an issue at the last step. Hence the cokernel is zero and we need only find the kernel.

We will build an isomorphism $f:\ker(\overline{can-\phi_p})\to W_{s}(R)$. For an element $x\in \ker(\overline{can-\phi_p})$, define $f(x):= x_{s-1}$. If $x_{s-1}=0$, then since $x\in \ker(\overline{can-\phi_p})$ it follows that $x_{s-2}=\phi_p^{-1}(can(x_{s-1}))=0$ and inductively $x=0$ so $f$ is injective. For surjectivity the element $x_{s-1}=a$ and inductively $x_{s-n-1}:=\phi_p^{-1}(can(x_{s-n}))$ will be a lift of $a$ to the kernel. 
\end{proof}

We now turn our attention to the case where $a'\mid m$.

\begin{lem}\label{lem: tc computation part 2}
Suppose that $b\nmid m'$ and $a'\mid m'$. Then the fiber of the map \[\prod_{\nu=0}^\infty \Sigma^{-1}\left(\thh(R;\ZZ_p)\wedge B_{m'p^\nu}\right)^{h\TT}\xrightarrow{can-\phi_p}\prod_{\nu=0}^\infty \Sigma^{-1}\left(\thh(R;\ZZ_p)\wedge B_{m'p^\nu}\right)^{t\TT}\] has homotopy groups in degree $2r$ given by $W_s(R)$ if $s\leq v_p(a)$ and $W_{s}(R)/V_{v_p(a)}W_{s-v_p(s)}(R)\cong W_{v_p(a)}(R)$ if $s>v_p(a)$. The odd homotopy groups vanish.
\end{lem}
\begin{proof}
We have a map of fiber sequences \[
\begin{tikzcd}
\prod\limits_{\nu=v_p(a)}^\infty\Sigma^{-1}(\thh(R;\ZZ_p)\wedge B_{m'p^\nu})^{h\TT} \arrow[d] \arrow[r, "can-\phi_p"]     & \prod\limits_{\nu=v_p(a)}^\infty\Sigma^{-1}(\thh(R;\ZZ_p)\wedge B_{m'p^\nu})^{t\TT} \arrow[d] \\
\prod\limits_{\nu=0}^\infty\Sigma^{-1}(\thh(R;\ZZ_p)\wedge B_{m'p^\nu})^{h\TT} \arrow[d] \arrow[r, "can-\phi_p"]          & \prod\limits_{\nu=0}^\infty\Sigma^{-1}(\thh(R;\ZZ_p)\wedge B_{m'p^\nu})^{t\TT} \arrow[d]      \\
\prod\limits^{\nu=v_p(a)-1}_{\nu=0}\Sigma^{-1}(\thh(R;\ZZ_p)\wedge B_{m'p^\nu})^{h\TT} \arrow[r, "\overline{can-\phi_p}"] & \prod\limits^{\nu=v_p(a)-1}_{\nu=0}\Sigma^{-1}(\thh(R;\ZZ_p)\wedge B_{m'p^\nu})^{t\TT}       
\end{tikzcd}
\]
and the factors in the bottom terms have $a,b\nmid m$. One of the top or bottom maps will be an isomorphism as we will soon see, and so the other map will contribute the homotopy groups. 

Suppose first that $s(a,b,r,p,m')\leq r$. By Corollary~\ref{cor: thh(R)wedge B_m} the factors in the top terms are given by cofibers \[\thh(R;\ZZ_p)\wedge B_m\simeq \Sigma^{2l(a,b,m)}\mathrm{cof}\left(\thh(R;\ZZ_p)\wedge (\TT/C_{m/a})_+\to \thh(R;\ZZ_p)\wedge (\TT/C_{m})_+\right)\] and we have maps of cofiber sequences \[\begin{tikzcd}
{\Sigma^{2l(a,b,m)}\thh(R;\ZZ_p)^{hC_{p^{\nu-v_p(a)}}}} \arrow[r] \arrow[d, "can"] & {\Sigma^{2l(a,b,m)}\thh(R;\ZZ_p)^{hC_{p^\nu}}} \arrow[r] \arrow[d, "can"] & \Sigma^{-1}\left(\thh(R;\ZZ_p)\wedge B_m\right)^{h\TT} \arrow[d, "can"] \\
{\Sigma^{2l(a,b,m)}\thh(R;\ZZ_p)^{tC_{p^{\nu-v_p(a)}}}} \arrow[r]                  & {\Sigma^{2l(a,b,m)}\thh(R;\ZZ_p)^{tC_{p^{\nu}}}} \arrow[r]                & \Sigma^{-1}\left(\thh(R;\ZZ_p)\wedge B_{m}\right)^{t\TT}               
\end{tikzcd}\]
and the left and middle maps are $2l(a,b,m)$-connected. Then upon taking $\pi_{2r-1}$ we have that the canonical map induces an isomorphism, and since the Frobenius increases the product index by 1 it follows that $can-\phi_p$ is an isomorphism on these terms. The even homotopy groups vanish by the calculations in Subsection~\ref{ssec: relative tc- and tp} and so these terms do not contribute to the homotopy groups of the fiber we are after in degrees $2r$ or $2r-1$. 

Thus if $s\leq v_p(a)$ the homotopy groups of the fiber of the map we are trying to compute in degrees $2r$ and $2r-1$ are given by the kernel and cokernel of  \[\prod\limits^{\nu=v_p(a)-1}_{\nu=0}\Sigma^{-1}(\thh(R;\ZZ_p)\wedge B_{m'p^\nu})^{h\TT} \xrightarrow{\overline{can-\phi_p}} \prod\limits^{\nu=v_p(a)-1}_{\nu=0}\Sigma^{-1}(\thh(R;\ZZ_p)\wedge B_{m'p^\nu})^{t\TT}  \] in degree $2r$. We are now in the same situation as in the previous lemma and so the cokernel of this map is zero and the kernel of this map is isomorphic to $W_s(R)$. 

Finally we consider the case when $s>v_p(a)$. In this case a similar argument as the above and in the previous lemma shows that the bottom map has zero cokernel and kernel isomorphic to $W_{v_p(a)}(R)$. We also have that 
\[\begin{tikzcd}
{\Sigma^{2l(a,b,m)}\thh(R;\ZZ_p)^{hC_{p^{\nu-v_p(a)}}}} \arrow[r] \arrow[d, "\phi_p^{hC_{p^{\nu-v_p(p)}}}"] & {\Sigma^{2l(a,b,m)}\thh(R;\ZZ_p)^{hC_{p^\nu}}} \arrow[r] \arrow[d, "\phi_p^{hC_{p^\nu}}"] & \left(\thh(R;\ZZ_p)\wedge B_m\right)^{h\TT} \arrow[d, "\phi_p"] \\
{\Sigma^{2l(a,b,m)}\thh(R;\ZZ_p)^{tC_{p^{\nu-v_p(a)+1}}}} \arrow[r]                                         & {\Sigma^{2l(a,b,m)}\thh(R;\ZZ_p)^{tC_{p^{\nu+1}}}} \arrow[r]                              & \left(\thh(R;\ZZ_p)\wedge B_{pm}\right)^{t\TT}                 
\end{tikzcd}\] 
is a map of cofiber sequences, where the bottom is a cofiber sequence and the Frobenius can be identified as such by the discussion of the Frobenius in Subsection~\ref{ssec: relative thh}. In particular we have the connectivity and coconnectivity bounds for the Frobenius that we need to make the above arguments work in this case as well. The only issue is in the base case $s=v_p(a)$: in the above we have been working with the product starting at $\nu=0$ and so $(-)^{tC_{p^0}}=0$ and the canonical map has been surjective. To fix this, note that the diagram \[\begin{tikzcd}
                                                                                                                                                                                           &  & {\Sigma^{2l(a,b,m'p^{v_p(a)})}\thh(R;\ZZ_p)^{tC_{p^{v_p(a)-v_p(a)}}}} \arrow[d, "V_{v_p(a)}"] \\
\Sigma^{-1}\left(\thh(R;\ZZ_p)\wedge B_{m'p^{v_p(a)-1}}\right)^{h\TT} \arrow[rr, "\phi_p^{hC_{p^{v_p(a)-1}}}"] \arrow[rrd, "\phi_p"] \arrow[rr, "{\simeq_{\geq 2l(a,b,m'p^{v_p(a)-1})}}"'] &  & {\Sigma^{2l(a,b,m'p^{v_p(a)})}\thh(R;\ZZ_p)^{tC_{p^{v_p(a)}}}} \arrow[d]                      \\
                                                                                                                                                                                           &  & \Sigma^{-1}\left(\thh(R;\ZZ_p)\wedge B_{m'p^{v_p(a)}}\right)^{t\TT}                          
\end{tikzcd}\] commutes and thus the Frobenius will still have the desired coconnectivity when you add in the factors $\nu<v_p(a)$. Hence we get a $W_s(R)/V_{v_p(a)}W_{s-v_p(a)}(R)$ in degree $2r$ and no homotopy in degree $2r-1$.
\end{proof}

\section{\texorpdfstring{The $K$-theory of the $p$-completed affine line}{The K-Theory of the p-completed affine line}}\label{sec: affine line}

We are now left only to compute $K(R\langle t\rangle, (t);\ZZ_p)$, or equivalently $K(R\langle t\rangle;\ZZ_p)$ up to $K(R;\ZZ_p)$. We will again use the Henselian invariance of $K^{inv}$, getting that $K^{inv}(R\langle t\rangle; \ZZ_p)\simeq K^{inv}(R/p[t];\ZZ_p)$. Thus we get the pullback square 
\[
\begin{tikzcd}
K(R\langle t\rangle;\ZZ_p) \arrow[r] \arrow[d] & \tc(R\langle t\rangle;\ZZ_p) \arrow[d] \\
{K(R/p[t];\ZZ_p)} \arrow[r]                    & {\tc(R/p[t];\ZZ_p)}                   
\end{tikzcd}
\]
which we may combine with the same Henselian invariance on the nilradical to get the pullback square
\begin{equation}\label{eqn: affine pullback round 1}
\begin{tikzcd}
 & K(R\langle t\rangle;\ZZ_p) \arrow[r] \arrow[d] & \tc(R\langle t\rangle;\ZZ_p) \arrow[d] \\
K(k;\ZZ_p)  \arrow[r,"\simeq"] & {K(k[t];\ZZ_p)} \arrow[r]                      & {\tc(k[t];\ZZ_p)}                     
\end{tikzcd}
\end{equation} where $k=(R/p)/\sqrt{0}$, which is in particular a perfect $\FF_p$-algebra. 

First note that since $R$ is perfectoid, it and $R[t]$ have bounded $p^{\infty}$-torsion. Thus we can use \parencite[Lemma 5.19]{Clausen_Mathew_Morrow} to get that $\tc(R\langle t\rangle; \ZZ_p)\simeq \tc(R[t];\ZZ_p)$. We will begin by computing this in the positive characteristic case. For this we will use a modification of the formula for $\tc^{-}(R\langle t^{\pm 1}\rangle ;\ZZ_p)$ and $\tp(R\langle t^{\pm 1}\rangle;\ZZ_p)$ found in \parencite[Lemma 3.3.4]{My_Thesis} along with the formula for the Frobenius and the canonical maps found in \parencite[Theorem 1.4.3]{My_Thesis}. These results are stated only for the ring $R=\ZZ_p^{cycl}$, but the only property of that ring used for the relevant calculations is that it is perfectoid. Since we are also not inverting $t$ for our setting and for the convenience of the reader we will record the details of these arguments here.

In order to compute the topological cyclic homology, we first need to compute the topological Hochschild homology. To this end note that we have an equivalence \[\thh(\mathbb{S}[t])\simeq \mathbb{S}\vee \bigvee_{n\in \ZZ_{\geq 1}}\Sigma^{\infty}_+\TT/C_n\] as Borel $\TT$ spectra. What is more, the Frobenius is given by the map \[\Sigma^\infty_+ \TT/C_n\to \left(\Sigma^\infty_+\TT/C_{np}\right)^{tC_{p}}\to \thh(\mathbb{S}[t])^{tC_{p}}\] where the first map is the $p^{\textrm{th}}$ power map on underlying spaces and in particular is a $p$-adic equivalence. For a recent reference for this fact, see \parencite[Example 2.2.10]{mccandless}. 

Since $\thh(-):\mathrm{Alg}_{\mathbb{E}_\infty}(\Sp)\to \mathrm{CycSp}$ is symmetric monoidal this then gives that \[\thh(R[t];\ZZ_p)\simeq \left(\thh(R;\ZZ_p)\vee\bigvee_{n\in \ZZ_{\geq 1}}\thh(R;\ZZ_p)\wedge (\TT/C_n)_+\right)^\wedge_p\] with the Frobenius given by the $p$-completion of the map which on the $n^{\textrm{th}}$ wedge summand is the composite \[\begin{tikzcd}
\thh(R;\ZZ_p)\wedge (\TT/C_n)_+ \arrow[rrdd, "\phi_p", bend right] \arrow[rr, "\phi_p^{\thh(R;\ZZ_p)}\wedge id"] &  & \thh(R;\ZZ_p)^{tC_{p}}\wedge (\TT/C_{n})_+ \arrow[d, "{id\wedge \phi_p^{\thh(\mathbb{S}[t])}}"] \\
                                                                                                                 &  & \thh(R;\ZZ_p)\wedge \left((\TT/C_{pn})_+\right)^{tC_{p}} \arrow[d, "l"]                         \\
                                                                                                                 &  & \left(\thh(R;\ZZ_p)\wedge (\TT/C_{pn})_+\right)^{tC_{p}}                                       
\end{tikzcd}\] where $l$ is the lax monoidal transformation for $(-)^{tC_{p}}$ which is an equivalence in this case by \parencite[Lemma 2]{Hesselholt_Nikolaus}.

\subsection{\texorpdfstring{Topological cyclic homology of the affine line for perfect $\FF_p$-algebras}{Topological cyclic homology of the affine line for perfect rings of positive characteristic}}\label{ssec: homotopy invariance charp}
In this section we will compute $\tc^{-}(R\langle t\rangle; \ZZ_p)$ and $\tp(R\langle t\rangle;\ZZ_p)$ for all perfectoid rings $R$. From this description, we will then be able to compute $\tc(R\langle t\rangle;\ZZ_p)$ for $R$ a perfect $\FF_p$ algebra. Unfortunately, while we will also get a long exact sequence for $\tc(R\langle t\rangle;\ZZ_p)$ in terms of these calculations for a general perfectoid ring, we were not able to explicitly calculate the kernel and cokernel of these maps and so need a different argument. The argument we make use of in Subsection~\ref{ssec: tc affine line general} need the positive characteristic case as input, and so we will handle that using these methods in this section.

Unlike in Section~\ref{sec: tc calculation}, we cannot rewrite the direct sum in our description of $\thh(R\langle t\rangle;\ZZ_p)$ as a product, and so we should not expect it to commute with homotopy fixed points or the Tate construction. It will, however, after completing with respect to the Nygaard filtration. This phenomenon will be explored in more detail in a forthcoming paper by the author.

\begin{lem}[Proposition 3.1.1, \cite{My_Thesis}]\label{lem: Nygaard completion of direct sum}
Let $X_i\in \Sp^{BG}$, $i\in I$ be a family of uniformly bounded below Borel $G$-equivariant spectra, where $G$ is any compact connected Lie group. Let $H_i$ be a set of closed subgroups of $G$. Then we have $k$-indexed filtrations
\begin{enumerate}
    \item \[\left(\bigvee_{i\in I}X_i\wedge \left(G/H_i\right)_+\right)_{hG}\simeq \lim_{k\to \infty}\bigvee_{i\in I}\left(\tau_{\leq k}X_i\right)_{hH_i};\]
    \item \[\left(\bigvee_{i\in I}X_i\wedge \left(G/H_i\right)_+\right)^{hG}\simeq \lim_{k\to \infty}\bigvee_{i\in I} \left(\Sigma^{T_{[H_i]}{G/H_i}}\tau_{\leq k}X_i\right)^{hH_i};\]
    \item and
    \[\left(\bigvee_{i\in I}X_i\wedge \left(G/H_i\right)_+\right)^{tG}\simeq \lim_{k\to \infty}\bigvee_{i\in I}\left(\Sigma^{T_{[H_i]}{G/H_i}}\tau_{\leq k}X_i\right)^{tH_i}.\]
\end{enumerate}
Here $\Sigma^{T_{[H_i]}{G/H_i}}$ is suspension at the $H_i$-representation on the tangent space of $G/H_i$ at $[H_i]$.
\end{lem}

The fact that homotopy orbits, fixed points, and the Tate construction can be written as a limit over the Postnikov truncations on the inside is due to the fact that taking homotopy orbits of a $k$-connective map is still $k$ connective. Rewriting the fixed points and Tate construction of the induced action in terms of the appropriate subgroup and representation sphere is again due to the Wirthm\"uller isomorphism, and so the content of this lemma is in the fact that for $Y_i$ uniformly bounded above and below we get that the homotopy fixed points will commute with direct sums.

\begin{lem}[Lemma 3.1.3, \cite{My_Thesis}]\label{lem: bounded above spectra and ht}
Let $Y_i$, $i\in I$ be a uniformly bounded (above and below) family of Borel $G$-spectra, $G$ a compact connected Lie group. Then \[\left(\bigvee_{i\in I}Y_i\right)^{hG}\simeq \bigvee_{i\in I}(Y_i)^{hG}.\]
\end{lem}
\begin{proof}
This argument will follow similarly to several results in \cite{Nikolaus_Scholze}. We will first consider the Eilenberg-MacLane spectra case. Assume that each $Y_i$ is an Eilenberg-MacLane spectrum $\Sigma^{n_i}M_i$. Since the homotopy groups of $\bigvee_{i\in I}Y_i$ are concentrated in finite degrees, this is a finite wedge of spectra with homotopy concentrated in a single degree. Homotopy fixed points will commute with the finite wedge sums, and we may therefore assume without loss of generality that $n_i=0$ for all $i\in I$.

Thus we get that $\bigvee_{i\in I}Y_i\simeq \bigoplus_{i\in I}M_i$, and the $G$-action on this spectrum must be trivial. To see this, note that $\operatorname{haut}(\bigoplus_{i\in I}M_i)$ is discrete as the automorphism space of an Eilenberg-MacLane spectrum, and so $B\operatorname{haut}(\bigoplus_{i\in I}M_i)$ is a $K(\pi,1)$ space. Then \[[BG, K(\pi,1)]=\hom(\pi_1(BG), \pi)/\{\textnormal{conjugation by elements of }\pi\}=0\] since $G$ is connected.

Since the action is trivial, we then have equivalences \[\left(\bigoplus_{i\in I}M_i\right)^{hG}\simeq \mathcal{F}_G(EG_+, \bigoplus_{i\in I}M_i)\simeq \mathcal{F}(BG_+, \bigoplus_{i\in I}M_i)\] which has homotopy groups $\pi_*\left(\mathcal{F}(BG_+, \bigoplus_{i\in I}M_i)\right)\cong H^{-*}(BG; \bigoplus_{i\in I}M_i)$. We may now apply the universal coefficient theorem to get the diagram \[\begin{tikzcd}
0 \arrow[d]                                                                         &  & 0 \arrow[d]                                                     \\
{\bigoplus_{i\in I}\mathrm{Ext}^1_{\ZZ}(H_{*-1}(BG;\ZZ), M_i)} \arrow[d] \arrow[rr, "\cong"] &  & {\mathrm{Ext}^1_{\ZZ}(H_{*-1}(BG;\ZZ), \bigoplus_{i\in I}M_i)} \arrow[d] \\
\bigoplus_{i\in I}H^*(BG;M_i) \arrow[d] \arrow[rr]                                  &  & H^*(BG;\bigoplus_{i\in I}M_i) \arrow[d]                         \\
{\bigoplus_{i\in I}\hom(H_*(BG;\ZZ), M_i)} \arrow[rr, "\cong"] \arrow[d]            &  & {\hom(H_*(BG;\ZZ), \bigoplus_{i\in I}M_i)} \arrow[d]            \\
0                                                                                   &  & 0                                   
\end{tikzcd}\]
where the columns are exact and the maps from the left to the right are induced by the comparison map in question. Since $G$ is compact, it has finitely generated homotopy groups, and then the classifying space fibration shows that $BG$ also has finitely generated homotopy groups. Hence $BG$ is also finite type, and the top and bottom horizontal maps in the above diagram are isomorphisms, and so the middle map is as well. In particular, the comparison map \[\bigvee_{i\in I}\left(M_i\right)^{hG}\to \left(\bigvee_{i\in I}M_i\right)^{hG}\] is a weak equivalence. 

In order to go from the Eilenberg-MacLane case to the general case, suppose the statement is true for spectra bounded below by $k$, and above by $k+n$. When $n=0$, this is the Eilenberg-MacLane case. Now, for $Y_i$ a family of spectra bounded below by $k$ and above by $k+n+1$ we have a fiber sequence \[\left(\bigvee_{i\in I}\Sigma^{n+k+1}\pi_{n+k+1}(Y_i)\right)^{hG}\simeq\left(\Sigma^{n+k+1}\pi_{n+k+1}\left(\bigvee_{i\in I}Y_i\right)\right)^{hG}\to \left(\bigvee_{i\in I}Y_i\right)^{hG}\to \left(\bigvee_{i\in I}\tau_{\leq n+k}Y_i\right)^{hG}.\] The sum commutes with homotopy fixed points in the left and the right terms, since the left term is Eilenberg-MacLane and the left term is the inductive hypothesis. Thus by naturality and the fiber sequence, the sum in the middle term also commutes with homotopy fixed points. By induction we are done
\end{proof}

From this we get the following.

\begin{lem}[Compare to Lemma 3.3.4 of \cite{My_Thesis}]\label{lem: tc- and tp for affine line}
There are equivalences \[\tc^{-}(R[t];\ZZ_p)\simeq  \tc^{-}(R;\ZZ_p)\vee \Sigma\left(\bigvee_{n\in \ZZ_{\geq 1}}\left(\thh(R;\ZZ_p)\right)^{hC_{n}}\right)^\wedge_{(p,d)}\]
and \[\tp(R[t];\ZZ_p)^\wedge_p\simeq \tp(R;\ZZ_p)\vee \Sigma\left(\bigvee_{n\in \ZZ_{\geq 1}}\left(\thh(R;\ZZ_p)\right)^{tC_{n}}\right)^\wedge_{(p,d)}\] where $R$ is a perfectoid ring, $d$ is an orientation of $R$.
\end{lem}
\begin{proof}
By Lemma~\ref{lem: Nygaard completion of direct sum} we have equivalences of the form \[\tc^{-}(R[t];\ZZ_p)\simeq \tc^{-}(R;\ZZ_p)\vee \Sigma\left(\lim_{k\to \infty}\bigvee_{n\in \ZZ_{\geq 1}}\left(\tau_{\leq 2k}\thh(R;\ZZ_p)\right)^{hC_n}\right)^\wedge_p\] and \[\tp(R[t];\ZZ_p)\simeq \tp(R;\ZZ_p)\vee \Sigma\left(\lim_{k\to \infty}\bigvee_{n\in \ZZ_{\geq 1}}\left(\tau_{\leq 2k}\thh(R;\ZZ_p)\right)^{tC_n}\right)^\wedge_p\] and so it is enough to show that the limit of the right hand terms amounts to $d$-adic completion.

To this end, note that $\tau_{\leq 2k}\thh(R;\ZZ_p)\simeq \thh(R;\ZZ_p)/u^k$ as a $\tc^{-}(R;\ZZ_p)$-module by \parencite[Theorem 6.1]{BMS2}. Then 
\begin{align*}
    \lim_{k\to \infty}\bigvee_{n\in \ZZ_{\geq 1}} (\tau_{\leq 2k}\thh(R;\ZZ_p))^{hC_{n}}    &\simeq \lim_{k\to \infty}\bigvee_{n\in \ZZ_{\geq 1}}(\thh(R;\ZZ_p)/u^k)^{hC_{p^n}}\\
                                &\simeq \lim_{k\to \infty}\left(\bigvee_{n\in \ZZ_{\geq 1}}\thh(R;\ZZ_p)^{hC_n}\right)/u^k\\
                                &\simeq \left(\bigvee_{n\in \ZZ_{\geq 1}}\thh(R;\ZZ_p)^{hC_n}\right)^\wedge_u
\end{align*} and similarly for the Tate construction. To see that this also agrees with the $d$-adic completion, it is enough to show that the maps $\left(\bigvee_{n\in \ZZ_{\geq 1}}\thh(R;\ZZ_p)^{hC_n}\right)/d^k\to \left(\bigvee_{n\in \ZZ_{\geq 1}}\thh(R;\ZZ_p)^{hC_n}\right)/u^k$ induced by the fact that $\tc^{-}(R;\ZZ_p)\to \tc^{-}(R;\ZZ_p)/u^k$ sends $d^k$ to zero have unbounded connectivity. For this note that it is enough to show this on each summand $\thh(R;\ZZ_p)^{hC_{n}}/d^k\to \thh(R;\ZZ_p)^{hC_n}/u^k$ uniformly. By \parencite[Corollary 5.8]{Me} we can rewrite this as the map \[(\tc^{-}(R;\ZZ_p)/d^k)/\Tilde{d}_{v_p(n)}v\to (\tc^{-}(R;\ZZ_p)/u^k)/\Tilde{d}_{v_p(n)}v\] and by the short exact sequences defining the outermost quotients it is then enough to show that \[\tc^{-}(R;\ZZ_p)/d^k\to \tc^{-}(R;\ZZ_p)/u^k\] has unbounded connectivity in $k$. This map is an equivalence below degree $k$ as can be seen from the homotopy fixed point spectral sequence (recall that $d$ is a generator of $E^2_{-2,2}$ by definition.). The Tate construction result follows similarly.
\end{proof}

We are now ready to calculate the topological cyclic homology of the affine line in positive characteristic.

\begin{lem}\label{lem: homotopy invariance for perfect fp algebras}
Let $k$ be a perfect $\FF_p$-algebra. Then the map \[\tc(k;\ZZ_p)\to \tc(k[t];\ZZ_p)\] is $(-2)$-truncated.
\end{lem}
\begin{proof}
By \parencite[Proposition II.1.9]{Nikolaus_Scholze} we have a fiber sequence \[\tc(k[t];\ZZ_p)\to \tc^{-}(k[t];\ZZ_p)\xrightarrow{\phi_p-can}\tp(k[t];\ZZ_p)\] and so we can use Lemma~\ref{lem: tc- and tp for affine line} to get the fiber sequence \[\tc(k[t];\ZZ_p)\to \tc^{-}(k;\ZZ_p)\vee \Sigma\left(\bigvee_{n\in \ZZ_{\geq 1}}\thh(k;\ZZ_p)^{hC_n}\right)^\wedge_p \xrightarrow{\phi-can}\tp(k;\ZZ_p)\vee \Sigma\left(\bigvee_{n\in \ZZ_{\geq 1}}\thh(k;\ZZ_p)^{tC_n}\right)^\wedge_p\] where $can$ respects the wedge sum grading, and $\phi_p$ multiplies the wedge sum grading by $p$. From the splitting $\tc(k;\ZZ_p)\to \tc(k[t];\ZZ_p)\to \tc(k;\ZZ_p)$ given by the inclusion $k\to k[t]$ followed by the map $t\mapsto 0$ we get that the map $\phi_p-can$ on the first summand has $\tc(k;\ZZ_p)$ as its fiber. Hence to prove this result it is enough to show that $\Sigma\left(\bigvee_{n\in \ZZ_{\geq 1}}\thh(k;\ZZ_p)^{hC_n}\right)^\wedge_p\xrightarrow{\phi_p-can}\Sigma\left(\bigvee_{n\in \ZZ_{\geq 1}}\thh(k;\ZZ_p)^{tC_n}\right)^\wedge_p$ is a weak equivalence in nonnegative degrees and an injection in degree $-1$. This map is then the suspension and $p$-completion of the map $\bigvee_{n\in \ZZ_{\geq 1}}\thh(k;\ZZ_p)^{hC_n}\to \bigvee_{n\in \ZZ_{\geq 1}}\thh(k;\ZZ_p)^{tC_n}$ which on the $n^{\textrm{th}}$ summand of the source is the map \[\thh(k;\ZZ_p)^{hC_n}\xrightarrow{\phi_p^{hC_n}\vee(-can)}\thh(k;\ZZ_p)^{tC_{pn}}\vee \thh(k;\ZZ_p)^{tC_n}\] which can be seen from \parencite[Proposition 3]{Hesselholt_Nikolaus}. 

It is now enough to show that the uncompleted map is an equivalence\footnote{This is the difference between the positive characteristic case versus the general case. When $d=p$ the map $\phi_p-can$ will be a $\ZZ_p$-module map and so its inverse will automatically be continuous. On the other hand the less $p$-torsion $R$ has the more linearly independent $d$ and $\phi^{-1}(d)$ will be, so it is not obvious to the author why the inverse we construct will be continuous.} in degrees nonnegative degrees and injective in degree $-2$. From \parencite[Section IV.4]{Nikolaus_Scholze} we have that $\phi_p$ is $(-1)$-coconnected, and since $\thh(k;\ZZ_p)_{hC_n}$ is connective it follows that $can$ is connective. Upon taking homotopy groups and using the same decomposition $m=m'p^\nu$ as described at the beginning of Subsection~\ref{ssec: relative tc} we are then getting  \[\bigoplus_{\nu =0}^{\infty}\bigoplus_{m\in J_p}W_{\nu+1}(k)\to \bigoplus_{\nu=0}^\infty \bigoplus_{m\in J_p}W_\nu(k)\] in nonnegative even homotopy groups, \[\bigoplus_{\nu=0}^\infty\bigoplus_{m\in J_p}W_\nu(k)\to \bigoplus_{\nu=0}^\infty\bigoplus_{m\in J_p}W_\nu(k)\] in negative even homotopy groups, and zero in odd homotopy groups. The map $\phi_p$ shifts $\nu$ up by one and does not change the internal direct sum, and can does not change either direct sum degree. We will show that in positive degrees this map is an isomorphism, the argument for the negative groups being very similar.

In nonnegative degrees the map $\phi_p$ induces an equivalence from the $\nu$-summand on the source to the $\nu+1$-summand in the target. Suppose that $x\in \ker(\phi_p-can)$ in a given degree. For a given $t\in \ZZ_{\geq 0}$ let $(-)_t:\bigoplus_{\nu=0}^\infty \bigoplus_{m\in J_p}W_{\nu+1}(k)\to \bigoplus_{m\in J_p}W_{t+1}(k)$ be the usual projection map. Since we are working with direct sums, there exists a $t\in \ZZ_{\geq 0}$ such that for all $\nu> t$, $(x)_\nu =0$. If $x\neq 0$, then there also exists a $t$ as above with $x_t\neq 0$. Then since $\phi_p$ is an isomorphism we see that $\phi_p(x)_{t+1}\neq 0$, but then $(\phi_p(x)-can(x))_{t+1}=\phi_p(x_t)-can(x_{t+1})=\phi_p(x_t)\neq 0$, a contradiction. Thus $\phi_p-can$ is injective. 

Now let $y\in \bigoplus_{\nu=0}^\infty\bigoplus_{m\in J_p}W_\nu(k)$, and let $t\in \ZZ_{\geq 0}$ be such that $y_\nu =0$ for all $\nu>t$, and $y_t\neq 0$. If no such $t$ exists then $y=0$ and has pre-image $0$. Notice also that for $y\neq 0$ that $t\geq 1$ since the $\nu=0$ summand in the target is zero. Inductively we define $x\in \bigoplus_{\nu=0}^\infty \bigoplus_{m\in J_p}W_{\nu+1}(k)$ by $x_\nu=0$ for all $\nu\geq t$ and $x_{t-s}=\phi_p^{-1}(y_{t-s+1}+can(x_{t-s+1}))$. Then in any given degree $(\phi_p(x)-can(x))_{t-s+1}=\phi_p(x_{t-s})-can(x_{t-s+1})=y_{t-s+1}+can(x_{t-s+1})-can(x_{t-s+1})=y_{t-s+1}$ and so $x$ is a pre-image of $y$ and $\phi_p-can$ is surjective.
\end{proof}

\begin{rem}
For any ring the topological cyclic homology will always be concentrated in degrees greater than or equal to $-1$. While the above argument might seem to suggest that in these degrees $N\tc(k[t])$ vanishes, this is before $p$-completion and in fact the homotopy group in degree $-2$ will have a non-zero derived contribution to $N\tc_{-1}(k[t])$. This might seem a bit mysterious from the perspective taken in this paper, but using topological restriction homology instead shows that these elements are coming from the cokernel of the map $W(k[t])\xrightarrow{F-1}W(k[t])$ which do not come from the cokernel of $W(k)\to W(k)$. From this point of view the fact that there are nonzero elements exactly corresponds to the fect that you cannot write $t\in k[t]$ as $x^p-x$ for any $x\in k[t]$. This was pointed out to me by Elden Elmanto.
\end{rem}

\subsection{Decompleting topological negative cyclic and periodic homology}\label{ssec: homotopy invariance charz}
In this Subsection we will play the description in Lemma~\ref{lem: tc- and tp for affine line} and the BMS filtrations from \cite{BMS2} off each other to get computations in the mixed characteristic case. First we will use Lemma~\ref{lem: tc- and tp for affine line} to get a computation of the homotopy groups of the topological negative cyclic homology and topological periodic homology in the $p$-torsion free case.

\begin{lem}\label{lem: refined tc- and tp computation for affine line}
Let $R$ be a $p$-torsion free perfectoid ring with orientation $d$. Then there are isomorphisms \[N\tc^{-}_{2k}(R[t];\ZZ_p)\cong N\tp_{2k}(R[t];\ZZ_p)\cong 0\] for all $k\in \ZZ$. There are also isomorphisms \[N\tc^{-}_{2k+1}(R[t];\ZZ_p)\cong \left(\bigoplus_{j\geq 1}A_{inf}(R)/d\cdot \Tilde{d}_{v_p(j)}\right)^\wedge_{(p,d)}\] for $k\geq 0$, \[N\tc^{-}_{2k+1}(R[t];\ZZ_p)\cong \left(\bigoplus_{j\geq 1}A_{inf}(R)/\Tilde{d}_{v_p(j)}\right)^\wedge_{(p,d)}\] for all $k<0$, and \[N\tp_{2k+1}(R[t];\ZZ_p)\cong \left(\bigoplus_{j\geq 1}A_{inf}(R)/\Tilde{d}_{v_p(j)}\right)^\wedge_{(p,d)}\] for all $k\in \ZZ$.
\end{lem}

\begin{proof}
From the description in Lemma~\ref{lem: tc- and tp for affine line} it is enough to show that the $A_{inf}(R)$-modules showing up before the completion have no higher derived completions. For this it is enough to show that the $A_{inf}(R)$ modules $\bigoplus_{j\geq 1}A_{inf}(R)/d\cdot \Tilde{d}_{v_p(j)}$ and $\bigoplus_{j\geq 1}A_{inf}(R)/\Tilde{d}_{v_p(j)}$ have bounded $d^{\infty}$ torsion and are $p$-torsion free. 

For the second statement it is enough to show that $A_{inf}(R)/d\cdot \Tilde{d}_n$ and $A_{inf}(R)/\Tilde{d}_n$ have no $p$-torsion for all $n$. To this end suppose that $x\in A_{inf}(R)/d\cdot \Tilde{d}_n$ is $p$-torsion. Let $y$ be a lift of $x$ in $A_{inf}(R)$ so that $py=zd\Tilde{d}_n$ for some $z\in A_{inf}(R)$. In particular $py=0\mod d$, but $A_{inf}(R)/d\cong R$ which by assumption has no $p$-torsion. Consequently $y=dy_1$ for some $y_1\in A_{inf}(R)$. Since $A_{inf}(R)$ has no $d$-torsion it follows that $py_1=z\Tilde{d}_n$. Inductively we have that $y=d\Tilde{d}_n y_n$ and so $x=0$. The same argument works for $A_{inf}(R)/\Tilde{d}_n$, so there is no $p$-torsion in any of the homotopy groups. 

It remains to show that $\bigoplus_{j\geq 1}A_{inf}(R)/\Tilde{d}_{v_p(j)}$ and $\bigoplus_{j\geq 1}A_{inf}(R)/d\Tilde{d}_{v_p(j)}$ have bounded $d^\infty$ torsion. For this it is enough to show that $A_{inf}(R)/\Tilde{d}_n$ and $A_{inf}(R)/d\Tilde{d}_n$ have uniformly bounded $d^\infty$ torsion. In fact $A_{inf}(R)/\Tilde{d}_n$ has no $d$ torsion and $A_{inf}(R)/d\Tilde{d}_n[d^\infty]=A_{inf}(R)/d\Tilde{d}_n[d]$.

To see this, recall that $\phi(x)=x^p+p\delta(x)$, and that for $d$ an orientation $\delta(d)$ is a unit by \parencite[Remark 3.11]{BMS1}. In addition by the ideal characterization of distinguished elements in \cite[Lemma 2.25]{BS} each $\phi^k_p(d)$ are distinguished and $\delta(\phi^k(d))$ is a unit for all $k>0$. Therefore we get that mod $d$ we have the equality
\begin{align*}
    \Tilde{d}_n &= \phi(d)\phi^2(d)\ldots \phi^n(d)\\
                &= (p\delta(d))(\delta(d)^p p^p+p\delta(\phi(d)))\ldots ((\phi^{n-1}(d))^p+p\delta(\phi^{n-1}(d)))\\
                &= up^n+p^{n-1+p}k 
\end{align*}
for some unit $u\in R$ and element $k\in R$. Then since $p\in \mathrm{rad}(R)$ by $p$-completeness it follows that $up^n+kp^{n-1+k}=vp^n$ for a unit $v\in R$. In particular $\Tilde{d}_n\neq 0 \mod d$. In fact since $R$ has no $p$-torsion we even have that there is no $\Tilde{d}_n$ torsion in $R$, which is a slight generalization of \cite[Lemma 3.6]{ALB}.

The above combines to show that the kernel of the map $A_{inf}(R)/d\xrightarrow{\Tilde{d}_n} A_{inf}(R)/d$ is zero. This in turn is isomorphic to $\mathrm{Tor}_1^{A_{inf}(R)}(A_{inf}(R)/d, A_{inf}(R)/\Tilde{d}_n)$, which is equivalent to the kernel of multiplication by $d$ on $A_{inf}(R)/\Tilde{d}_n$. by induction the $d^{\infty}$ torsion of $A_{inf}(R)/\Tilde{d}_n$ is then zero. 

For the positive degree homotopy groups, note that for $xd^k=0\mod d\cdot \Tilde{d}_n$ implies that for a lift $y$ of $x$ in $A_{inf}(R)$, there exists a $z$ such that $yd^k=zd\Tilde{d}_n$. Since $d$ is a non-zero divisor, this implies that $yd^{k-1}=z\Tilde{d}_n$. The previous paragraph then shows that $z=z'd^{k-1}$ and $y=z'\Tilde{d}_n$. Consequently all the $d^k$ torsion lies in the ideal $(\Tilde{d}_n)$ and conversely all the elements in this ideal are $d$ torsion. Hence the $d$ and $d^\infty$ torsion are both $(\Tilde{d}_n)$ as desired.
\end{proof}

We will now shift focus to the BMS filtration on these theories. For this it will he helpful to have some notation. Define $\DD_{R\langle t\rangle, (t)/A_{inf}(R)}$ denote the fiber of the map $\DD_{R\langle t\rangle/A_{inf}(R)}\xrightarrow{t\mapsto 0}\DD_{R/A_{inf}(R)}$ and similarly $\mathcal{N}^{\geq i}\DD_{R\langle t\rangle, (t)/A_{inf}(R)}$ is the fiber of the map $\mathcal{N}^{\geq i} \DD_{R\langle t\rangle/A_{inf}(R)}\to \mathcal{N}^{\geq i}\DD_{R/A_{inf}(R)}$ for all $i$. Notice that these fiber sequences are in fact split by the map $R\hookrightarrow R\langle t\rangle$ and so in particular any (co)connectivity bounds for $R\langle t\rangle$ will remain true for these relative constructions.

\begin{lem}
Let $R$ be a perfectoid ring. Then $\DD_{R\langle t\rangle/ A_{inf}(R)}$ is in homological degrees $[-1,0]$.
\end{lem}
\begin{proof}
By the Hodge-Tate comparison of \cite[Theorem 4.10]{BS} we have isomorphisms $H^i(\DD_{R\langle t\rangle/A_{inf}(R)}/\Tilde{d})\cong \Omega^i_{R\langle t\rangle/R}\{-i\}$. Since we are working over a base perfectoid ring $R$, there is a (noncanonical) isomorphism $\Omega^i_{R\langle t\rangle/R}\{-i\}\cong \Omega^i_{R\langle t\rangle/R}$ and therefore $H^i(\DD_{R\langle t\rangle/A_{inf}(R)}/\Tilde{d})=0$ for $i>1$. Consequently $\Tilde{d}:\DD_{R\langle t\rangle/A_{inf}(R)}\to \DD_{R\langle t\rangle/A_{inf}(R)}$ is an isomorphism in homological degrees $\leq 3$ and surjective in homological degree $-2$. By derived $\Tilde{d}$-adic completeness the result follows.
\end{proof}

\begin{cor}\label{cor: prismatic cohomology is em}
Let $R$ be a $p$-torsion free perfectoid ring. Then $\DD_{R\langle t\rangle, (t)/A_{inf}(R)}\{k\}$ is an Eilenberg-MacLane spectrum in degree $-1$ for all integers $k$.
\end{cor}

\begin{proof}
For notational convenience we will denote $H^i(\DD_{R\langle t\rangle, (t)/A_{inf}(R)}\{j\})$ as $H^i(\DD\{j\})$. We then have by \cite[Theorem 1.12(4)]{BMS2} a spectral sequence of the form \[E_2^{s,t}=H^{s-t}(\DD\{-t\})\implies N\tp_{-s-t}(R;\ZZ_p)\] with homological Serre grading on the differential. The $E_2$ page is given below:
\begin{center}
\begin{sseqpage}[x range ={-3}{2}, y range = {-2}{2}, homological Serre grading, classes = {draw = none }, xscale=2, x axis extend end = 4em]
\class["0"](-3,2)
\class["0"](-2,2)
\class["0"](-1,2)
\class["0"](0,2)
\class["0"](1,2)
\class["H^0(\DD\{-2\})"](2,2)
\class["0"](-3,1)
\class["0"](-2,1)
\class["0"](-1,1)
\class["0"](0,1)
\class["H^0(\DD\{-1\})"](1,1)
\class["H^1(\DD\{-1\})"](2,1)
\class["0"](-3,0)
\class["0"](-2,0)
\class["0"](-1,0)
\class["H^1(\DD)"](1,0)
\class["H^0(\DD)"](0,0)
\class["0"](2,0)
\class["0"](-3,-1)
\class["0"](-2,-1)
\class["0"](1,-1)
\class["0"](2,-1)
\class["H^1(\DD\{1\})"](0,-1)
\class["H^0(\DD\{1\})"](-1,-1)
\class["0"](-3,-2)
\class["0"](0,-2)
\class["0"](1,-2)
\class["0"](2,-2)
\class["H^1(\DD\{2\})"](-1,-2)
\class["H^0(\DD\{2\})"](-2,-2)
\d2(0,-1)
\end{sseqpage}
\end{center}

By the previous Lemma we know that $H^i(\DD\{j\})=0$ for all $i\neq 0,1$. Consequently the spectral sequence collapses for degree reasons. We therefore see that $N\tp_{2k}(R;\ZZ_p)\cong H^0(\DD\{k\})$ and so by the previous Lemma we see that $H^0(\DD\{k\})=0$ as desired.
\end{proof}

Our ultimate goal is to get connectivity bounds on the complexes $N\ZZ_p(i)(R)$. In order to do this we will also need similar connectivity bounds on the Nygaard filtered pieces.

\begin{lem}
Let $R$ be a perfectoid ring. Then the complexes $\mathcal{N}^{\geq i}\DD_{R\langle t\rangle, (t)/A_{inf}(R)}$ are concentrated in homological degrees $[-2,0]$ for all $i\in \ZZ$. If $R$ is in addition $p$-torsion free perfectoid ring then The complexes $\mathcal{N}^{\geq i}\DD_{R\langle t\rangle, (t)/A_{inf}(R)}$ are concentrated in homological degrees $[-2,-1]$ for all $i\in \ZZ$.
\end{lem}

\begin{proof}
For $i\leq 0$ we have that $\mathcal{N}^{\geq i}\DD_{R\langle t\rangle, (t)/A_{inf}(R)}=\DD_{R\langle t\rangle, (t)/A_{inf}(R)}$ and so this is a weaker version of the above Corollary. We will proceed by induction.

Suppose that for some $i\in \NN$ that $\mathcal{N}^{\geq i}\DD_{R\langle t\rangle, (t)/A_{inf}(R)}$ is concentrated in degrees $[-2,0]$ (or $[-2,-1]$ in the $p$-torsion free case.). We then have a cofiber sequence \[\mathcal{N}^{\geq i+1}\DD_{R\langle t\rangle, (t)/A_{inf}(R)}\to \mathcal{N}^{\geq i}\DD_{R\langle t\rangle, (t)/A_{inf}(R)}\to \mathcal{N}^{i}\DD_{R\langle t\rangle, (t)/A_{inf}(R)}\] and $\mathcal{N}^i \DD_{R\langle t\rangle, (t)/A_{inf}(R)}$ is concentrated in degrees $[-1,0]$ by \cite[Corollary 6.10]{BMS2}. The long exact sequence on homotopy groups then gives that 
\[
\begin{tikzcd}
0 \arrow[r] & {\pi_0\left(\mathcal{N}^{\geq i+1}\DD_{R\langle t\rangle, (t)/A_{inf}(R)}\right)} \arrow[r]    & {\pi_0\left(\mathcal{N}^{\geq i}\DD_{R\langle t\rangle, (t)/A_{inf}(R)}\right)} \arrow[r]    & {\pi_0\left(\mathcal{N}^{i}\DD_{R\langle t\rangle, (t)/A_{inf}(R)}\right)} \arrow[ld,out=-30]      \\
            &                                                                                                & \ldots  \arrow[ld, out=-30]                                                                           &                                                                                            \\
            & {\pi_{-2}\left(\mathcal{N}^{\geq i+1}\DD_{R\langle t\rangle, (t)/A_{inf}(R)}\right)} \arrow[r] & {\pi_{-2}\left(\mathcal{N}^{\geq i}\DD_{R\langle t\rangle, (t)/A_{inf}(R)}\right)} \arrow[r] & {\pi_{-2}\left(\mathcal{N}^{i}\DD_{R\langle t\rangle, (t)/A_{inf}(R)}\right)} \arrow[lldd, out=-30] \\
            &                                                                                                &                                                                                              &                                                                                            \\
            & {\pi_{-3}\left(\mathcal{N}^{\geq i+1}\DD_{R\langle t\rangle, (t)/A_{inf}(R)}\right)} \arrow[r] & {\pi_{-3}\left(\mathcal{N}^{\geq i}\DD_{R\langle t\rangle, (t)/A_{inf}(R)}\right)} \arrow[r] & \ldots                                                                                    
\end{tikzcd}
\]
This long exact sequence immediately shows the general statement, so we will focus now on the $p$-torsion free statement.

The top line and the inductive hypothesis show that $\pi_0(\mathcal{N}^{\geq i+1}\DD_{R\langle t\rangle, (t)/A_{inf}(R)})=0$. The inductive hypothesis and the bottom line forces $\pi_{-3}\left(\mathcal{N}^{\geq i+1}\DD_{R\langle t\rangle, (t)/A_{inf}(R)}\right)=0$. The other homotopy groups are also zero by this long exact sequence, so by induction we are done.
\end{proof}

\begin{cor}\label{cor: nygaard peices are em}
Let $R$ be a $p$-torsion free perfectoid ring. Then the complexes $\mathcal{N}^{\geq i}\DD_{R\langle t\rangle, (t)/A_{inf}(R)}\{i\}$ are Eilenberg-MacLane spectra in degree $-1$.
\end{cor}

\begin{proof}
For notational convenience we will denote $H^j(\mathcal{N}^{\geq i}\DD_{R\langle t\rangle, (t)/A_{inf}(R)}\{i\})$ as $H^j(\mathcal{N}^{\geq i})$. We then have by \cite[Theorem 1.12(4)]{BMS2} a spectral sequence of the form \[E_2^{s,t}=H^{s-t}(\mathcal{N}^{\geq -t})\implies N\tc^{-}_{-s-t}(R;\ZZ_p)\] with homological Serre grading on the differential. The $E_2$ page is given below:
\begin{center}
\begin{sseqpage}[x range ={-3}{2}, y range = {-2}{2}, homological Serre grading, classes = {draw = none }, xscale=2, x axis extend end = 4em]
\class["0"](-3,2)
\class["0"](-2,2)
\class["0"](-1,2)
\class["0"](0,2)
\class["0"](1,2)
\class["0"](2,2)
\class["0"](-3,1)
\class["0"](-2,1)
\class["0"](-1,1)
\class["0"](0,1)
\class["0"](1,1)
\class["H^1(\mathcal{N}^{\geq -1})"](2,1)
\class["0"](-3,0)
\class["0"](-2,0)
\class["0"](-1,0)
\class["H^1(\DD)"](1,0)
\class["0"](0,0)
\class["H^2(\mathcal{N}^{\geq 0})"](2,0)
\class["0"](-3,-1)
\class["0"](-2,-1)
\class["H^2(\mathcal{N}^{\geq 1})"](1,-1)
\class["0"](2,-1)
\class["H^1(\mathcal{N}^{\geq 1})"](0,-1)
\class["0"](-1,-1)
\class["0"](-3,-2)
\class["H^2(\mathcal{N}^{\geq 2})"](0,-2)
\class["0"](1,-2)
\class["0"](2,-2)
\class["H^1(\mathcal{N}^{\geq 2})"](-1,-2)
\class["0"](-2,-2)
\d2(0,-1)
\end{sseqpage}
\end{center}

By the previous Lemma we know that $H^j(\mathcal{N}^{\geq i})=0$ for all $i\neq 1,2$. Consequently the spectral sequence collapses for degree reasons. We therefore see that $N\tc^{-}_{2k-2}(R;\ZZ_p)\cong H^2(\mathcal{N}^{\geq k})$ and so by Lemma~\ref{lem: refined tc- and tp computation for affine line} we have that $H^2(\mathcal{N}^{\geq i})=0$ for all $i\in \ZZ$ as desired. 
\end{proof}

We now have enough information to consider the BMS filtration on topological cyclic homology.

\begin{lem}\label{lem: nzps are em sorta}
Let $R$ be a $p$-torsion free perfectoid ring. Then the complexes $N\ZZ_p(i)(R)$ are concentrated in homological degrees $[-2,-1]$ and $N\ZZ_p(1)(R)$ is Eilenberg-MacLane concentrated in degree $-1$.
\end{lem}
\begin{proof}
The general homological bounds are due to Corollary~\ref{cor: prismatic cohomology is em} and Corollary~\ref{cor: nygaard peices are em} and the defining fiber sequence of $N\ZZ_p(i)(R)$. We also note that when $R$ is not $p$-torsion free we still have by the defining fiber sequence that $N\ZZ_p(i)(R)$ is contained in homological degrees $[-2,0]$. To see that connectivity bound on $N\ZZ_p(1)(R)$ we will in fact show that $\ZZ_p(1)(R\langle t\rangle )$ is contained in homological degrees $[-1,0]$. 

From \cite[Theorem 5.2]{antieau2021beilinson} we have that the fiber of the map $\ZZ_p(1)(R\langle t\rangle)\to \ZZ_p(1)(R/\sqrt{(p)}[t])$ have homotopy concentrated in degrees $\geq -1$. In particular $\pi_{-k}(\ZZ_p(1)(R\langle t\rangle))\to \pi_{-k}(\ZZ_p(1)(R/\sqrt{(p)}[t]))$ is an isomorphism for all $k\geq 2$. We will show that $\pi_{-2}(N\ZZ_p(1)(R/\sqrt{(p)}))=0$. To this end consider the spectral sequence of \cite[Theorem 1.12(5)]{BMS2} which is in our case of the form $E^{s,t}_2=H^{s-t}(N\ZZ_p(-t)(R/\sqrt{(p)}))\implies N\tc_{-s-t}(R/\sqrt{(p)})$. For notational convinience in the spectral sequence we will denote $H^j(N\ZZ_p(i)(R/\sqrt{(p)}))$ as $H^{j}(i)$. The $E_2$ is recorded below:
\begin{center}
\begin{sseqpage}[x range ={-3}{2}, y range = {-2}{2}, homological Serre grading, classes = {draw = none }, xscale=2, x axis extend end = 4em]
\class["0"](-3,2)
\class["0"](-2,2)
\class["0"](-1,2)
\class["0"](0,2)
\class["0"](1,2)
\class["H^0(-2)"](2,2)
\class["0"](-3,1)
\class["0"](-2,1)
\class["0"](-1,1)
\class["0"](0,1)
\class["H^0(-1)"](1,1)
\class["H^1(-1)"](2,1)
\class["0"](-3,0)
\class["0"](-2,0)
\class["0"](-1,0)
\class["H^1(0)"](1,0)
\class["H^0(0)"](0,0)
\class["H^2(0)"](2,0)
\class["0"](-3,-1)
\class["0"](-2,-1)
\class["H^2(1)"](1,-1)
\class["0"](2,-1)
\class["H^1(1)"](0,-1)
\class["H^0(1)"](-1,-1)
\class["0"](-3,-2)
\class["H^2(2)"](0,-2)
\class["0"](1,-2)
\class["0"](2,-2)
\class["H^1(2)"](-1,-2)
\class["H^0(2)"](-2,-2)
\d2(0,-1)
\end{sseqpage}
\end{center}
For degree reasons this spectral sequence collapses and among other things we see that $N\tc_0(R/\sqrt{(p)};\ZZ_p)\cong H^1(N\ZZ_p(1)(R/\sqrt{(p)}))$. The left hand term is then zero by Lemma~\ref{lem: homotopy invariance for perfect fp algebras}. The result follows.
\end{proof}

We are now ready to prove the main result of this Subsection.
\begin{prop}\label{prop: Even vanishing}
Let $R$ be a $p$-torsion free perfectoid ring. Then $N\tc_0(R;\ZZ_p)=0$. If in addition $R$ is a $\ZZ_p^{cycl}$ algebra then $N\tc(R;\ZZ_p)$ has $2$ periodic homotopy groups in degrees $\geq 0$ and so all nonnegative even homotopy groups will also vanish.
\end{prop}

\begin{proof}
To see the statement for $N\tc_0(R;\ZZ_p)$ consider the BMS spectral sequence from \cite[Theorem 1.12(5)]{BMS2} which in our case is given by $E_2^{s,t}=H^{s-t}(N\ZZ_p(-t)(R))\implies N\tc_{-s-t}(R;\ZZ_p)$ with $E_2$ page given below:
\begin{center}
\begin{sseqpage}[x range ={-3}{2}, y range = {-2}{2}, homological Serre grading, classes = {draw = none }, xscale=2, x axis extend end = 9em]
\class["0"](-3,2)
\class["0"](-2,2)
\class["0"](-1,2)
\class["0"](0,2)
\class["0"](1,2)
\class["0"](2,2)
\class["0"](-3,1)
\class["0"](-2,1)
\class["0"](-1,1)
\class["0"](0,1)
\class["0"](1,1)
\class["H^1(-1)"](2,1)
\class["0"](-3,0)
\class["0"](-2,0)
\class["0"](-1,0)
\class["H^1(0)"](1,0)
\class["0"](0,0)
\class["H^2(0)"](2,0)
\class["0"](-3,-1)
\class["0"](-2,-1)
\class["0"](1,-1)
\class["0"](2,-1)
\class["H^1(1)"](0,-1)
\class["0"](-1,-1)
\class["0"](-3,-2)
\class["H^2(2)"](0,-2)
\class["0"](1,-2)
\class["0"](2,-2)
\class["H^1(2)"](-1,-2)
\class["0"](-2,-2)
\draw[name = zeroline, color=red, rotate=-26.5] (-0.2,1.8) ellipse (2.5 and 0.5);
\node[right, color=red] at (2.2,-2) {N\tc_0(R;\ZZ_p)};
\end{sseqpage}
\end{center}

Here $H^j(i)$ denotes $H^j(N\ZZ_p(i)(R))$, and the degree bounds produced in this section are included in the diagram above. From this we see that there is nothing on the $s=-t$ line(circled above in red) and so $N\tc_0(R;\ZZ_p)\cong 0$.

We now turn to the case of $R$ a $p$-torsion free perfectoid $\ZZ_p^{cycl}$ algebra. In this case we can be slightly more specific about some of the objects appearing in the constructions above. First recall the notation $\epsilon := (1,\zeta_p, \zeta_{p^2},\ldots)\in R^\flat$ where $\zeta_{p^k}$ is a primitive $p^k$ root of unity. Then by \cite[Example 3.16]{BMS1} we have that $d:= 1+[\epsilon^{{1}/{p}}]+[\epsilon^{{1}/{p}}]^2+\ldots +[\epsilon^{{1}/{p}}]^{p-1}$ is an orientation of $R$ where $[-]:R^\flat \to A_{inf}(R)$ is the Tiechm\"uller lift. If we let $\mu:= [\epsilon]-1$ we then also have that $\Tilde{d}=\phi(\mu)/\mu$ and that $d=\mu/\phi^{-1}(\mu)$ by \cite[Proposition 3.17]{BMS1}. We then have that the element $\phi^{-1}(\mu)u\in \tc^{-}_{2}(R;\ZZ_p)$ is such that $can(\phi^{-1}(\mu)u)=d\phi_{-1}(\mu)\sigma = \mu\sigma=\phi_p(\phi^{-1}(\mu)u)$ in $\tp_{2}(R;\ZZ_p)$ by \cite[Proposition 6.2 and 6.3]{BMS2}. In particular by \cite{Nikolaus_Scholze} there exists a unique element $\beta\in \tc_2(R;\ZZ_p)$ lifting $\phi^{-1}(\mu)u$. 

We will show that multiplication by $\beta$ is an isomorphism on $N\tc(R;\ZZ_p)$ in degrees $\geq 0$. Base changing the fiber sequence of \cite[Proposition4.3]{Nikolaus_Scholze} gives a cofiber sequence \[N\tc(R;\ZZ_p)[\beta^{-1}]^\wedge_p\to N\tc^{-}(R;\ZZ_p)[\beta^{-1}]^\wedge_p\to N\tp(R;\ZZ_p)[\beta^{-1}]^\wedge_p\] and we have comparison maps from the cofiber sequence giving $N\tc(R;\ZZ_p)$. On homotopy groups the action of $\beta$ on $N\tc^{-}(R;\ZZ_p)$ is given by multiplication by $\phi^{-1}(\mu)u$ and on $N\tp(R;\ZZ_p)$ it is given by multiplication by $\mu\sigma$. 

We will first show that in fact $\tc^{-}(R;\ZZ_p)[\beta^{-1}]^\wedge_p\simeq \tc^{-}(R;\ZZ_p)[d^{-1}]^\wedge_p$ and that $\tp(R;\ZZ_p)[\beta^{-1}]^\wedge_p\simeq \tp(R;\ZZ_p)[d^{-1}]^\wedge_p$. To see this first note that we have canonical comparison maps in both cases from the right to the left given by the fact that $d=\phi^{-1}(\mu)^{p-1}\mod p$ and $\mu=\phi^{-1}(\mu)^p\mod p$ and so inverting $\beta$ in the $p$-complete sense will also invert $d$, $\phi^{-1}(\mu)$, and $\mu$ as well. Thus it is enough to check that this map is an isomorphism on homotopy groups, but then on homotopy groups we have that \[\pi_*\left(\tc^{-}(R;\ZZ_p)[\beta^{-1}]^\wedge_p\right)=\pi_*\left(\colim \ldots\to\Sigma^{2}\tc^{-}(R;\ZZ_p)\xrightarrow{\times \beta}\tc^{-}(R;\ZZ_p)\to\ldots\right)^\wedge_p\] and the map $\beta$ on homotopy groups is injective. Thus as a filtered colimit of $p$-torsion free groups it the colimit of $\tc_*^{-}(R;\ZZ_p)$ under multiplication by $\beta$ will still be $p$-torsion free. Consequently there will be no higher derived $p$-completions and $\pi_*\left(\tc^{-}(R;\ZZ_p)[\beta^{-1}]^\wedge_p\right)\cong \left(A_{inf}(R)[u,v]/(uv-d)[\beta^{-1}]\right)^\wedge_p$. To see that this is the same as inverting $d$ in the $p$-complete sense notice that $u\mid d$ and $d=\phi^{-1}(\mu)^{p-1}\mod p$ imply that $\beta$ is invertable after inverting $d$ in the $p$-complete sense, and that as already discussed inverting $\beta$ will also invert $d$. The argument for the topological periodic homology is similar. 

To see that $\beta$ acts isomorphically on $N\tc_k(R;\ZZ_p)$ for $k\geq 0$ it is now enough to show that $N\tc(R;\ZZ_p)\to N\tc(R;\ZZ_p)[\beta^{-1}]^\wedge_p$ is $(-2)$-truncated, which by the above is then equivalent to show that the total fiber of the square 
\[
\begin{tikzcd}
N\tc^{-}(R;\ZZ_p) \arrow[r] \arrow[d] & {N\tc^{-}(R;\ZZ_p)[d^{-1}]^\wedge_p} \arrow[d] \\
N\tp(R;\ZZ_p) \arrow[r]               & {N\tp(R;\ZZ_p)[d^{-1}]^\wedge_p}              
\end{tikzcd}
\]
has homotopy groups concentrated in degrees $\leq -2$. To see this consider the completion maps \[\left(\bigvee_{j\in \ZZ_{\geq 1}}\Sigma\thh(R;\ZZ_p)^{hC_j}\right)^\wedge_p\to N\tc^{-}(R;\ZZ_p)\] and \[\left(\bigvee_{j\in \ZZ_{\geq 1}}\Sigma\thh(R;\ZZ_p)^{tC_j}\right)^\wedge_p\to N\tp(R;\ZZ_p)\] which by Lemma~\ref{lem: tc- and tp for affine line} takes the form of $d$-adic completion of the left hand side. Note that modulo $p$ the left hand side of both of these maps are $d$-torsion, and so by an arithmetic fracture square argument the cofiber of these maps are $N\tc^{-}(R;\ZZ_p)[d^{-1}]^\wedge_p$ and $N\tp(R;\ZZ_p)[d^{-1}]^\wedge_p$, respectively. The maps are also the usual comparison maps, and so the horizontal fibers of the square in question are given by the $p$-completed but not $d$-completed wedge sum. 

It remains to show that that map \[\left(\bigvee_{j\in \ZZ_{\geq 1}}\Sigma\thh(R;\ZZ_p)^{hC_j}\right)^\wedge_p\to \left(\bigvee_{j\in \ZZ_{\geq 1}}\Sigma\thh(R;\ZZ_p)^{tC_j}\right)^\wedge_p\] is $(-2)$-truncated. To see this in degrees $\geq 0$ the argument is the same as in Lemma~\ref{lem: homotopy invariance for perfect fp algebras}. We now need only show that the map on $\pi_{0}$ is surjective and the map on $\pi_{-1}$ is injective. For the surjectivity note that $\pi_*\left(\bigvee_{j\in \ZZ_{\geq 1}}\Sigma\thh(R;\ZZ_p)^{tC_j}\right)$ is $p$-torsion free and so will have no higher derived $p$-completions. In particular $\pi_0\left(\bigvee_{j\in \ZZ_{\geq 1}}\Sigma\thh(R;\ZZ_p)^{tC_j}\right)^\wedge_p=0$ and so we do not need to worry about surjectivity. 

To see that in degree $-1$ this map is injective, unwinding definitions this is equivalent to the map \[\left(\bigoplus_{j\in \ZZ_{+}}A_{inf}(R)/\Tilde{d}_{v_p(j)}\right)^\wedge_p\xrightarrow{1-F}\left(\bigoplus_{j\in \ZZ_{+}}A_{inf}(R)/\Tilde{d}_{v_p(j)}\right)^\wedge_p\] is injetive where $F$ is some map $A_{inf}(R)/\Tilde{d}_{v_p(j)}\to A_{inf}(R)/\Tilde{d}_{v_p(pj)}$ which up to an automorphism of $A_{inf}(R)/\Tilde{d}_{v_p(pj)}$ is induced by the Frobenius\footnote{It turns out that this automorphism is the identity, but this takes some work to show and is ultimately unnecessary to our argument.}. Since each $A_{inf}(R)/\Tilde{d}_{v_p(j)}$ is $p$-torsion free and already $p$-complete we then have that \[\left(\bigoplus_{j\in \ZZ_{+}}A_{inf}(R)/\Tilde{d}_{v_p(j)}\right)^\wedge_p\hookrightarrow \prod_{j\in \ZZ_{+}}A_{inf}(R)/\Tilde{d}_{v_p(j)}\] is injective. Suppose that $x\in \ker(1-F:\prod_{j\in \ZZ_{+}}A_{inf}(R)/\Tilde{d}_{v_p(j)}\to \prod_{j\in \ZZ_{+}}A_{inf}(R)/\Tilde{d}_{v_p(j)})$. Either $x=0$ or there is a smallest $j$ such that $x_j\neq 0$, in which case $(1-F)(x)_j=x_j\neq 0$ since $F$ multiplies the product index by $p$ and $x_{j/p}=0$ by assumption. This second case contradicts $x\in \ker(1-F)$ and so $x=0$. Thus $1-F$ on the product is injective and so by the $5$-Lemma the map in question is also injective.
\end{proof}

\begin{rem}
The idea for this proof came from the observation of Mathew in \cite[Propositio 5.10]{k1_local_tr} that for a smooth algebra over a perfectoid base topological restriction homology is asymptotically $K(1)$ local and if in addition the perfectoid ring you start with has a consistent system of $p^{\textrm{th}}$ power roots then the $K$-theory is a $ku^\wedge_p$-module by \cite[Example 5.5]{k1_local_tr}. We believe that $\beta$ is coming from the action of the Bott class via this module structure, but have not checked this. While it would be simpler to cite these results, we need the periodicity to start by degree $0$. The cited result only kicks in after degree $2$.  
\end{rem}

From the periodicity we are also able to give a description of the odd homotopy groups.

\begin{cor}\label{cor: odd k group calculation}
Let $R$ be a $p$-torsion free perfectoid $\ZZ_p^{cycl}$-algebra. Then for $r\geq 0$ we have a short exact sequence \[0\to \operatorname{Ext}(\ZZ/p^\infty, 1+\sqrt{pR_0}\langle t\rangle)\to K_{2r+1}(R\langle t\rangle,(t);\ZZ_p)\to T_p(\pic(R_0\langle t\rangle))\to 0\] where $1+\sqrt{pR}\langle t\rangle\subset R\langle t\rangle$ and $T_p(G)=\hom(\ZZ/p^\infty, G)$ is the Tate module. This sequence is splittable and the first term can be identified with $\left(1+\sqrt{pR_0}\langle t\rangle\right)^\wedge_p$.
\end{cor}
\begin{proof}
From Subsection\ref{ssec: homotopy invariance charp} and pullback square~\ref{eqn: affine pullback round 1} we have that \[K(R\langle t\rangle;\ZZ_p)\simeq \tau_{\geq 0}\tc(R\langle t\rangle;\ZZ_p)\] and so we may use the results of this Subsection. In particular the $K$-groups will be $2$-periodic and we need only show this for $r=0$.

From the above we also have that $K_1(R\langle t\rangle, (t);\ZZ_p)\cong H^1(N\ZZ_p(1)(R))$ and so it is enough to identify this cohomology group. For this we will use \cite[Proposition 7.17]{BMS2} in the form that Bhatt and Lurie write it in \cite[Lemma 7.5.7]{bhatt_lurie}. From this we have a description of $H^1(\ZZ_p(1))$ as the global sections of the derived $p$-completion of $\mathbb{G}_m$. In particular \cite[Proposition 2.5]{Bousfield_localization} gives splittable short exact sequences \[0\to \operatorname{Ext}(\ZZ/p^\infty, R^\times)\to H^1(\ZZ_p(1)(R))\to T_p(H^1(R;\mathbb{G}_m))\to 0\] and \[0\to \operatorname{Ext(\ZZ/p^\infty, R\langle t\rangle^\times)}\to H^1(\ZZ_p(R\langle t\rangle))\to T_p(H^1(R\langle t\rangle;\mathbb{G}_m))\to 0\] which are functorial in their input. 

We have an isomorphism $H^1(R;\mathbb{G}_m)\cong \pic(R)$, and by \cite[Corollary 9.7]{BS} this is a uniquely $p$-divisible group. Thus $T_p(\pic(R))=0$ and by combining the two short exact sequences above we get a short exact sequence \[0\to \operatorname{Ext}(\ZZ/p^\infty, R\langle t\rangle^\times /R^\times) \to H^1(N\ZZ_p(1)(R)))\to T_p(\pic(R\langle t\rangle) \to 0\] which is still splittable since the splitting of $H^1(\ZZ-p(1)(R))$ is natural since the third term vanishes. All that is left to do is identify the first term. In order to do this it is enough to compute $R\langle t\rangle^\times$ and to show that $R\langle t\rangle^\times / R^\times$ has bounded $p^\infty$-torsion. 

Note that $(R\langle t\rangle, \sqrt{pR\langle t\rangle})$ is a Henselian pair, so an element $x(t)\in R\langle t\rangle$ is a unit if and only if its image is a unit in $R\langle t\rangle /\sqrt{pR\langle t\rangle}=R/\sqrt{pR}[t]$. The ring $R/\sqrt{pR}$ is reduced (by definition) so the only units in the polynomial ring are the constant terms. We also have that $(R, \sqrt{pR})$ is Henselian, so the constant term of $x(t)$ is a unit in $R$ if and only if it is a unit in $R/\sqrt{pR}$. Consequently $x(t)\in R\langle t\rangle$ is a unit if and only if it is of the form $u+\sqrt{pR}t$.

To see that $R\langle t\rangle^\times/R^\times$ has bounded $p^\infty$-torsion, let $x(t)\in 1+\sqrt{pR}\langle t\rangle$ be a $p^k$-torsion element for some $k\geq 1$. Write $x(t)=1+x_1t+x_2t^2+\ldots$ and notice then the the equation $x(t)^{p^k}=1$ translates to \[1=1+p^kx_1t+o(t^2)\] and in particular $p^kx_1=0$. Since $R$ is $p$-torsion free we have that $x_1=0$. Inductively $x_i=0$ for all $i\geq 1$ and so $x(t)=1$ and $R\langle t\rangle^\times/R^\times$ in fact has no $p$-power torsion.
\end{proof}
\subsection{Topological cyclic homology of the affine line over general perfectoid base}\label{ssec: tc affine line general}

From the Subsection~\ref{ssec: homotopy invariance charp} and pullback square~\ref{eqn: affine pullback round 1} we see that for $R$ a perfectoid ring \[K(R\langle t\rangle; \ZZ_p)\simeq \tau_{\geq 0}\tc(R\langle t\rangle; \ZZ_p)\] To show Theorem~\ref{thm: Formal affine invariance} it is then equivalent to work with topological cyclic homology. To show the result, we prove it separately for the positive characteristic case and the $p$-torsion free case, and then use an excision argument to recover the result in general. Subsection~\ref{ssec: homotopy invariance charp} handled the positive characteristic case, and we will see that the results from Subsection~\ref{ssec: homotopy invariance charz} and the previous Subsection will handle the $p$-torsion free case. The excision result is then essentially just \parencite[Construction 7.20]{k1_local_tr}.

\begin{prop}\label{prop: pullback ptf and charp}
Let $R$ be a perfectoid ring. Then there is a pullback square \[\begin{tikzcd}
\tc(R\langle t\rangle;\ZZ_p) \arrow[d] \arrow[r] & \tc(R_0\langle t\rangle;\ZZ_p) \arrow[d] \\
{\tc(k[t];\ZZ_p)} \arrow[r]                      & {\tc(k_0[t];\ZZ_p)}                     
\end{tikzcd}\] where $R_0=R/R[p]$ is the original perfectoid ring mod the ideal of $p$-torsion elements, and $k$ and $k_0$ are perfect $\FF_p$-algebras.
\end{prop}
\begin{proof}
By \parencite[Construction 7.20 ]{k1_local_tr} we have a Milnor square 
\[\begin{tikzcd}
R \arrow[d] \arrow[r] & R_0 \arrow[d] \\
{k} \arrow[r]                      & {k_0}          
\end{tikzcd}\] with $R_0$, $k$, and $k_0$ as described in the statement of the proposition. In addition this diagram is also pushout in $\mathbb{E}_\infty$ rings. After applying $-\otimes_{R}R[ t]$ then gives a pushout diagram \[\begin{tikzcd}
R[t] \arrow[d] \arrow[r] & R_0[t] \arrow[d] \\
{k[t]} \arrow[r]                      & {k_0[t]}          
\end{tikzcd}\] of $\mathbb{E}_\infty$ rings. In addition the vertical maps are still surjections and the induced map on the kernels of the vertical maps will still be an isomorphism since adding an independent variable has no higher $\mathrm{Tor}$ groups. Then \parencite[Theorem 2.7]{Land_Tamme} gives that \[\begin{tikzcd}
\tc(R[ t];\ZZ_p) \arrow[d] \arrow[r] & \tc(R_0[ t];\ZZ_p) \arrow[d] \\
{\tc(k[t];\ZZ_p)} \arrow[r]                      & {\tc(k_0[t];\ZZ_p)}                     
\end{tikzcd}\] is a pullback square. Since $R[t]$ and $R_0[t]$ have bounded $p^\infty$-torsion, the result follows.
\end{proof}

Since we handled the case when $p=0$ in the previous section, assume without loss of generality that $p\neq 0$ in $R$. Then the map $\tc(k[t];\ZZ_p)\to \tc(k_0[t];\ZZ_p)$ will restrict to the identity on $\tc(\FF_p[t];\ZZ_p)$, and so it is an isomorphism on $\pi_0$. Consequently the map $\tc(R\langle t\rangle;\ZZ_p)\to \tc(R_0\langle t\rangle;\ZZ_p)$ is an isomorphism in positive degrees and an injection on $\pi_0$. 

We are now ready to prove the truncation result.

\begin{proof}[Proof of Theorem~\ref{thm: Formal affine invariance}]
First let us consider the case when $(n,p)=1$. For this we might as well assume that $n=l^i$ for some prime $l\neq p$ and $i\geq 1$. In this case Gabber rigidity \cite{Gabber_Rigidity} appies and \[K(R\langle t\rangle; \ZZ/l^i)\simeq K(k[t];\ZZ/l^i)\] where $k:= (R/p)/\sqrt{0}$ is a perfect $\FF_p$-algebra. Once $p$ is nilpotent in our ring we have an equivalence $K(k;\ZZ/l^i)\simeq K(k[t];\ZZ/l^i)$ by \parencite[Consequence 1.1]{Weibel_homotopy_invariance}.

Now suppose that $R$ is a $p$-torsion free perfectoid $\ZZ_p^{cycl}$-algebra. By Proposition~\ref{prop: Even vanishing} we have that $N\tc_{2k}(R;\ZZ_p)=0$ for all $k$, and Corollary~\ref{lem: nzps are em sorta} and the BMS spectral sequence and $2$-periodicity we have that $N\tc_{2k+1}(R;\ZZ_p)\cong H^1(R\langle t\rangle; N\ZZ_p(1))$. By Corollary~\ref{cor: odd k group calculation} we then have that $N\tc_{2k+1}(R;\ZZ_p)\cong T_p\pic(R\langle t\rangle)\oplus (1+\sqrt{pR}\langle t\rangle)^\wedge_p$ as desired.

Finally let $R$ be a general perfectoid $\ZZ_p^{cycl}$-algebra. If $p=0$ in $R$ then $R$ is a perfect $\FF_p$ algebra and $R\langle t\rangle=R[t]$. For $p\neq 0$ in $R$ we then have from Proposition~\ref{prop: pullback ptf and charp} and Lemma~\ref{lem: homotopy invariance for perfect fp algebras} that for $R_0=R/R[p]$ that $\tc(R\langle t\rangle; \ZZ_p)\to \tc(R_0\langle t\rangle;\ZZ_p)$ is an isomorphism in positive degrees and an injection on $\pi_0$. The result then follows from the previous paragraph along with the fact that $N\tc_0(R_0\langle t\rangle;\ZZ_p)=0$ so since it injects into the zero group $N\tc_0(R\langle t\rangle)=0$.
\end{proof}

\section{Final computations}~\label{sec: final computation}
In this section we will put all the work we have done in the previous sections together and prove Theorem~\ref{thm: main computation} and Theorem~\ref{thm: main computation p-completed}.
\begin{proof}[Proof of Theorem~\ref{thm: main computation}]
From Section~\ref{sec: main pullback squares} we have a pullback square of the form \[\begin{tikzcd}
K(A) \arrow[r] \arrow[d] & \tc(A) \arrow[d] \\
{K(R[t])} \arrow[r]      & {\tc(R[t])}     
\end{tikzcd}\] and so upon profinite completion a pullback square \[\begin{tikzcd}
K(A;\widehat{\ZZ}) \arrow[r] \arrow[d] & \tc(A;\widehat{\ZZ}) \arrow[d] \\
{K(R[t];\widehat{\ZZ})} \arrow[r]      & {\tc(R[t];\widehat{\ZZ})}     
\end{tikzcd}\]
From Lemma~\ref{lem: tc calculation part 1} and Lemma~\ref{lem: tc computation part 2} we know the homotopy groups of the vertical fiber after $p$-completion, they are as described in the statement of Theorem~\ref{thm: main computation}. We also have by assumption that $K(R;\widehat{\ZZ})\to K(R[t];\widehat{\ZZ})$ is an equivalence, and so it remains to show that $\tc(A,R[t];\ZZ_p)\to \tc(A,R[t];\widehat{\ZZ})$ is an equivalence. 

For this it is enough to show that for $l\neq p$ a prime, $\tc(A,R[t];\ZZ_l)=0$, and this in turn is implied by the vanishing of $\tc(A;\ZZ_l)$ and $\tc(R[t];\ZZ_l)$. Since $l$-adic completion commutes with limits it is enough to show that $\thh(A;\ZZ_l)=\thh(R[t];\ZZ_l)=0$. Recall that for a prime $l$, the $l$-adic completion of a spectrum $X$ is given by the Bousfield localization $L_{\mathbb{S}/l}X$ of $X$ at the Moore spectrum $\mathbb{S}/l$. Both $A$ and $R[t]$ have the property that $L_{\mathbb{S}/l}X=0$ since $l$ is a unit in both rings, and so $L_{\mathbb{S}/l}(X^{\wedge n+1})=L_{\mathbb{S}/l}(L_{\mathbb{S}/l}(X)\wedge X^{\wedge n})=0$ for both rings. Since $L_{\mathbb{S}/l}$ is a left adjoint it commutes with colimits and \[\thh(X;\ZZ_l)\simeq L_{\mathbb{S}/l}|L_{\mathbb{S}/l}(X^{\wedge \bullet+1})|=0\] for $X$ either $A$ or $R[t]$, as desired.  
\end{proof}

The proof of Theorem~\ref{thm: main computation p-completed} will be similar, but we will need to handle the primes away from $p$ slightly differently.

\begin{proof}[Proof of Theorem~\ref{thm: main computation p-completed}]
We will first show that $K(A^\wedge_p, (x,y);\ZZ_l)\simeq 0$ for $l\neq p$. For this we can apply Gabber rigidity \cite{Gabber_Rigidity} to see that \[K(A^\wedge_p, (x,y);\ZZ_l)\simeq K(k[x,y]/(y^a-x^b), (x,y);\ZZ_l)\] where $k:=R/\sqrt{pR}$. In particular $k$ is a perfect $\FF_p$-algebra and so this is a consequence of \cite{Hesselholt_Madsen}. 

To handle the case when $l=p$, we will use the pullback square~\ref{eqn: main pullback square}. We now consider the commutative diagram below:
\[
\begin{tikzcd}
K(A^\wedge_p;\ZZ_p) \arrow[rd] \arrow[rr] &                                       & K(R;\ZZ_p) \\
                                          & K(R\langle t\rangle;\ZZ_p) \arrow[ru] &           
\end{tikzcd}
\]
The homotopy groups of the fiber of the left diagonal map, which as a spectrum is $K(A^\wedge_p, R\langle t\rangle;\ZZ_p)$, are given by the Witt vector expression in the statement of Theorem~\ref{thm: main computation p-completed}. The homotopy groups of the fiber of the right diagonal map are given by $T_p\pic(R_0)$ in odd positive degrees and zero otherwise by Theorem~\ref{thm: Formal affine invariance}. Then by the octahedral axiom the fiber of the horizontal map, which is $K(A^\wedge_p, (x,y);\ZZ_p)$, sits in a fiber sequence \[K(A^\wedge_p, R\langle t\rangle;\ZZ_p)\to K(A^\wedge_p, (x,y);\ZZ_p)\to K(R\langle t\rangle, (t);\ZZ_p)\] and the long exact sequence of homotopy groups of this sequence then gives the result.
\end{proof}

\printbibliography
\end{document}